\documentclass[12pt]{amsart}
\usepackage{amsmath, amssymb, amsfonts}
\usepackage[all]{xypic}
\usepackage[pdftex]{graphicx}
\usepackage[usenames]{color}
\usepackage{stmaryrd}
\usepackage[utf8]{inputenc}
\usepackage{bbold}
\usepackage{mathtools}

\theoremstyle{plain}
\newtheorem{theorem}{Theorem}[section]
\newtheorem{lemma}[theorem]{Lemma}
\newtheorem{corollary}[theorem]{Corollary}
\newtheorem{proposition}[theorem]{Proposition}

\theoremstyle{remark}
\newtheorem{remark}[theorem]{Remark}

\newtheorem{example}[theorem]{Example}

\newtheorem*{note*}{Note}
\newtheorem*{remark*}{Remark}
\newtheorem*{example*}{Example}

\theoremstyle{definition}
\newtheorem*{definition*}{Definition}
\newtheorem{definition}[theorem]{Definition}

\newcommand{\Fitt}{\mathrm{Fitt}}

\newcommand{\Ann}{\mathrm{Ann}}
\newcommand{\GL}{\mathrm{GL}}

\newcommand{\Hom}{\mathrm{Hom}}
\newcommand{\Aff}{\mathrm{Aff}}
\newcommand{\Nrd}{\mathrm{Nrd}}
\newcommand{\Aut}{\mathrm{Aut}}
\newcommand{\End}{\mathrm{End}}
\newcommand{\Quot}{\mathrm{Quot}}
\newcommand{\id}{\mathrm{id}}
\newcommand{\Ext}{\mathrm{Ext}}
\newcommand{\cl}{\mathrm{cl}}
\newcommand{\Char}{\mathrm{Char}}
\newcommand{\aug}{\mathrm{aug}}

\newcommand{\Z}{\mathbb{Z}}

\newcommand{\Q}{\mathbb{Q}}

\newcommand{\N}{\mathbb{N}}
\newcommand{\F}{\mathbb{F}}

\renewcommand{\det}{\mathrm{det}}
\numberwithin{equation}{section}

\title[Noncommutative Fitting invariants]{Notes on 
	Noncommutative Fitting invariants}

\author[Andreas Nickel]{Andreas Nickel\\ \\ 
	\tiny{with an appendix by Henri Johnston and Andreas Nickel}}

\address{
	Department of Mathematics\\
	University of Exeter\\
	Exeter\\
	EX4 4QF\\
	United Kingdom
}
\email{H.Johnston@exeter.ac.uk}
\urladdr{http://emps.exeter.ac.uk/mathematics/staff/hj241}

\address{Universit\"{a}t Duisburg--Essen\\
	Fakult\"{a}t f\"{u}r Mathematik\\
	Thea-Leymann-Str. 9\\
	45127 Essen\\
	Germany}
\email{andreas.nickel@uni-due.de}
\urladdr{https://www.uni-due.de/$\sim$hm0251/english.html}

\date{Version of 10th September 2018}
\begin{document}
	
\begin{abstract}
To each finitely presented module $M$ over a commutative ring $R$
one can associate an $R$-ideal $\Fitt_{R}(M)$, which is called the (zeroth) Fitting ideal of $M$ over $R$.
This is of interest because it is always contained in the $R$-annihilator $\Ann_{R}(M)$ of $M$, but is often much easier to compute.
This notion has recently been generalised to that of so-called `Fitting invariants' over certain noncommutative rings;
the present author considered the case in which $R$ is an $\mathfrak{o}$-order $\Lambda$ in a finite dimensional separable algebra,
where $\mathfrak{o}$ is an integrally closed commutative noetherian complete local domain.
This article is a survey of known results and open problems in this context. 
In particular, we investigate the behaviour of Fitting invariants under direct sums. 
In the appendix, we present a new approach to Fitting invariants via Morita equivalence.
\end{abstract}

\maketitle

\section*{Introduction}
Let $R$ be a commutative unitary ring and let $M$ be a finitely presented 
$R$-module. This means that there is an exact sequence
\begin{equation} \label{eqn:finite-presentation-comm}
R^a \stackrel{h}{\longrightarrow} R^b \longrightarrow
M \longrightarrow 0,
\end{equation}
where $a$ and $b$ are positive integers.
In other words, the module $M$ is finitely generated and there is
a finite number of relations between the generators.
This information is incorporated in the $a \times b$ matrix $h$.
Note that every finitely generated module over a noetherian ring
is indeed finitely presented.

Suppose that $a \geq b$. Then the (zeroth) Fitting ideal of $M$
over $R$ is defined to be the $R$-ideal generated by all $b \times b$
minors of the matrix $h$:
\[
\Fitt_R(M) := \langle \det(H) \mid H \in S_b(h) \rangle_R,
\]
where $S_b(h)$ denotes the set of all $b \times b$ submatrices of $h$.
In the case $a<b$ one simply puts $\Fitt_R(M) := 0$. This notion was
introduced by the German mathematician 
Hans Fitting \cite{Fitting-DMV} who showed that it is
in fact independent of the chosen finite presentation $h$.
Fitting was a student of Emmy Noether and is also famous for his
contributions to group theory.

Fitting ideals became an important tool in commutative algebra.
A key property is that the Fitting ideal of $M$ is always contained in the
$R$-annihilator ideal of $M$, but in many cases is much easier to
compute. For instance, it behaves well under epimorphisms,
certain exact sequences, and direct sums of
$R$-modules: if $M$ and $N$ are two finitely presented $R$-modules,
then one has an equality
\begin{equation} \label{eqn:additivity-comm}
	\Fitt_R(M \oplus N) = \Fitt_R(M) \cdot \Fitt_R(N).
\end{equation}
For a full account of the theory, we refer the reader to
\cite{MR0460383}.

Fitting ideals have many applications in number theory. 
We give two typical examples. If $L/K$
is a finite Galois extension of number fields with Galois group $G$,
then the class group $\cl_L$ of $L$ has a natural structure as a module
over the group ring $\Z[G]$. If $p$ is a prime then the
$p$-part $\Z_p \otimes_{\Z} \cl_L$ of the class group is a module
over the $p$-adic group ring $\Z_p[G]$.

Now assume that $p$ is odd and that $L/K$ is a CM-extension.
If $G$ is abelian then Greither \cite{MR2371374} has computed the
Fitting ideal of the Pontryagin dual of the minus part of
$\Z_p \otimes_{\Z} \cl_L$ via the equivariant Tamagawa number conjecture.
This gives strong evidence for a conjecture of Brumer that asserts that
certain `Stickelberger elements' (constructed from values at zero of Artin
$L$-functions attached to the irreducible characters of $G$)
annihilate the class group.

Fitting ideals also appear in Iwasawa theory. The formulation and the
proof of the (classical) main conjecture for
totally real fields by Wiles \cite{MR1053488} makes heavy use
of the close relation between characteristic ideals and Fitting ideals
of Iwasawa modules.

Fitting ideals of arithmetic objects are also interesting 
in their own right. Kurihara
and Miura \cite{MR2805636} showed that (away from the $2$-primary part) the 
Fitting ideal of the minus class group
of an absolutely abelian imaginary number field coincides with
the Stickelberger ideal (as conjectured by Kurihara
\cite{MR1998607}). This gives the precise arithmetical
interpretation of the latter ideal.

As Galois groups are in general non-abelian, it is natural to ask
whether analogous invariants can be defined for modules over
noncommutative rings such as $p$-adic group rings.
Grime considered several cases in his PhD thesis \cite{grime_thesis},
including matrix rings over commutative rings. We will describe an
approach to noncommutative Fitting invariants over rings that
are Morita equivalent to a commutative ring in the appendix.
This is essentially a generalisation of Grime's approach (and of
\cite[\S 2]{MR3092262}). Parker has treated the case of $p$-adic group rings
in his PhD thesis \cite{parker_thesis} under the 
(for applications sometimes too restrictive)
hypothesis that one can choose $a = b$ in 
\eqref{eqn:finite-presentation-comm}.

Now let $\mathfrak o$ be an integrally closed commutative noetherian
complete local domain with field of quotients $F$.
Let $\Lambda$ be an $\mathfrak o$-order in a finite dimensional separable
algebra $A$ over $F$. We will call such an order a Fitting order
over $\mathfrak o$. A standard example is that of $p$-adic group rings
$\Z_p[G]$ where $p$ is a prime and $G$ is a finite group.
The Iwasawa algebra $\Z_p \llbracket G \rrbracket$ of a one-dimensional
$p$-adic Lie group $G$ is a second example.

Let $\Lambda$ be an arbitrary Fitting order and let $M$ be a finitely
presented $\Lambda$-module. In \cite{MR2609173} the present author
defined the `maximal Fitting invariant' 
$\Fitt_{\Lambda}^{\max}(M)$ of $M$ over $\Lambda$ as an 
equivalence class of certain modules over the centre of $\Lambda$
using reduced norms. If $\Lambda$ is commutative then the reduced norm
coincides with the usual determinant and thus this notion is compatible
with that of the classical Fitting ideal.
This approach has been studied further by Johnston and the present author
in \cite{MR3092262}. It has also been applied in number theory
to study the class group and other Galois modules
in extensions with arbitrary Galois groups
\cite{MR2976321, MR3072281, MR2801311, non-abelian-zeta,
	MR3739317, non-abelian-Brumer-Stark}
(see also \cite{Brumer-Gross-Stark} for a survey).

In this article we do not use the notion of `$\Nrd(\Lambda)$-equivalence'
as in \cite{MR2609173}. We essentially follow the alternative approach in
\cite[\S 3.5]{MR3092262} and define Fitting invariants as a genuine ideal
of a certain commutative ring $\mathcal{I}(\Lambda)$ that contains the 
centre of $\Lambda$ and the reduced norms of every matrix with entries
in $\Lambda$. As long as we are only interested in annihilation results,
this approach has no disadvantage over the more complicated notion of
$\Nrd(\Lambda)$-equivalence. The latter is only necessary when one
wishes to relate Fitting invariants to (relative) $K$-theory.

In order to obtain annihilators from $\Fitt_{\Lambda}^{\max}(M)$
one has to multiply by a certain ideal in
$\mathcal{I}(\Lambda)$ which we call the denominator ideal.
In this article we report on basic properties of noncommutative
Fitting invariants and on lower bounds for the denominator ideal.
In particular, we consider the case where $\Lambda$ is a $p$-adic 
group ring. Inspired by \cite[Lemma 6]{MR2046598} and
recent work of Kataoka \cite{Kataoka},
we give a new shorter and simpler proof of a proposition in
\cite{MR2609173} on the behaviour of Fitting invariants 
under Pontryagin duality.
Finally, in section \S \ref{sec:additivity} we investigate whether 
Fitting invariants are additive in the sense of \eqref{eqn:additivity-comm}
when working over noncommutative rings. 
We show that this property does hold for certain classes of 
Fitting orders but, perhaps surprisingly, that it does not hold for
non-maximal hereditary Fitting orders over complete discrete valuation rings.

In the appendix, we present a new approach to Fitting invariants
for unitary rings $\Lambda$ which are Morita equivalent to
a commutative ring. This generalises \cite[\S 2]{MR3092262},
where the case of matrix rings over commutative rings is considered.

In many cases we will not provide rigorous proofs. Instead we
try to motivate the results and give many examples.

\subsection*{Acknowledgements}
The author acknowledges financial support provided by the 
Deutsche Forschungsgemeinschaft (DFG) 
within the Heisenberg programme (No.\, NI 1230/3-1).
I am indebted to Masato Kurihara, Kenichi Bannai
and Takeshi Tsuji for the excellent organisation of the Iwasawa 2017
conference, where I had the opportunity to give a Preparatory Lecture
Series on `Non-abelian Stark-type conjectures and noncommutative Iwasawa
theory' including most of the material presented in this article.
I also thank Henri Johnston, Takenori Kataoka and David Watson
for fruitful discussions
and help concerning various aspects of Fitting invariants.

\subsection*{Notation and conventions}
All rings are assumed to have an identity element and all modules are assumed
to be left modules unless otherwise stated.
%Unadorned tensor products will always denote tensor products over $\Z$.
%If $K$ is a field, we denote its absolute Galois group by $G_K$.
If $\Lambda$ is a ring, we write $M_{m \times n}(\Lambda)$ for the set of all 
$m \times n$ matrices with entries in $\Lambda$. 
We denote the group of invertible
matrices in $M_{n \times n}(\Lambda)$ by $\GL_n(\Lambda)$
and let $\boldsymbol{1}_n \in \GL_n(\Lambda)$ be the $n \times n$
identity matrix.
Moreover, we let $\zeta(\Lambda)$ denote the centre of the ring $\Lambda$.
We shall sometimes abuse notation by using the symbol $\oplus$ to denote
the direct product of rings or orders.
%If $M$ is an $R$-module we denote by $\pd_R(M)$ the projective dimension
%of $M$ over $R$.

\section{The commutative case}
The material presented in this section originates with Fitting
\cite{Fitting-DMV}. We refer the reader to \cite{MR0460383}
and \cite[\S 20]{MR1322960} for further details.

Let $R$ be a commutative ring and let $M$ be a finitely presented
$R$-module. Choose a finite presentation $h$ as in 
\eqref{eqn:finite-presentation-comm} and define
\[
\Fitt_{R}(M) := \left\{\begin{array}{lll}
0 & \mathrm{ if } & a < b\\
\langle \det(H) \mid H \in S_b(h) \rangle_{R}
& \mathrm{ if } & a \geq b,
\end{array}\right.
\]
where we recall that $S_b(h)$ denotes the set of all $b \times b$
submatrices of $h$.
We call $\Fitt_R(M)$ the \emph{Fitting ideal} of $M$ over $R$.

\begin{theorem} \label{thm:Fitt-well-comm}
	Let $R$ be a commutative ring and let $M$ be a finitely presented 
	$R$-module. Then $\Fitt_R(M)$ is independent of the choice of $h$
	and thus is well-defined.
\end{theorem}

The idea of the proof is as follows
(for a full proof see \cite[\S 3.1]{MR0460383} or \cite[\S 20.2]{MR1322960}). 
Since two ideals are equal
if and only if they become equal in every localisation of $R$,
we can and do assume that $R$ is local. We choose $b' \in \N$ minimal
such that there is a surjection $R^{b'} \twoheadrightarrow M$.
Similarly, we choose $a' \in \N$ minimal such that there is a
finite presentation
\[
R^{a'} \stackrel{h'}{\longrightarrow} R^{b'} \longrightarrow
M \longrightarrow 0.
\]
It suffices to show that the Fitting ideals coming from $h$ 
and $h'$ coincide.
We may view \eqref{eqn:finite-presentation-comm} as the truncation of
a free resolution $\mathcal{F}$ of $M$. Similarly,
we may view $h'$ as a truncation of a \emph{minimal} free resolution 
$\mathcal{F}'$ of $M$.  However, as $R$ is local,
every free resolution is isomorphic to the direct sum of
$\mathcal F'$ and a trivial complex (see \cite[Theorem 20.2]{MR1322960}).
In particular, there are invertible matrices $X \in \GL_a(R)$ and
$Y \in \GL_b(R)$ such that
\[
h \circ X = Y \circ \left( \begin{array}{ccc}
h' & 0 & 0 \\ 0 & \boldsymbol{1} & 0
\end{array}\right),
\]
where $\boldsymbol{1} = \boldsymbol{1}_{b-b'}$ 
is a $(b-b') \times (b-b')$ identity matrix.
As $\det(Y)$ belongs to $R^{\times}$, we may assume that $Y=1$.
By an exercise in linear algebra that uses the multilinearity of
the determinant we may likewise assume that $X=1$.
The result now follows easily.

\begin{example} \label{ex:R/I}
	Let $I = \langle r_1, \dots, r_a \rangle_R$ 
	be a finitely generated ideal of $R$ and
	take $M = R/I$. Then we have a finite presentation
	\[
	R^a \longrightarrow R \longrightarrow R/I \longrightarrow 0,
	\]
	where the first arrow maps the $i$-th standard basis vector of $R^a$ to
	$r_i$, $1 \leq i \leq a$. Thus we have that
	\[
	\Fitt_R(R/I) = I.
	\]
\end{example}

\begin{remark} \label{rem:base-change}
	Let $R \rightarrow S$ be a homomorphism of commutative rings.
	As taking tensor products is a right exact functor, a finite
	presentation \eqref{eqn:finite-presentation-comm} of the
	$R$-module $M$ yields a finite presentation
	\[
		S^a \longrightarrow S^b \longrightarrow S \otimes_R M 
		\longrightarrow 0
	\]
	of $S \otimes_R M$.
	Fitting ideals therefore commute with base change:
	\[
		\Fitt_S (S \otimes_R M) = S \otimes_R \Fitt_R(M).
	\]
\end{remark}

\begin{remark}
	Let $M$ be a finitely presented $R$-module with a finite presentation
	$h$ as in \eqref{eqn:finite-presentation-comm}.
	For each integer $i \geq 0$ one can define `higher Fitting ideals'
	$\Fitt_R^i(M)$ of $M$ that are generated by the minors of $h$
	of size $b-i$. These invariants form an increasing sequence
	\[
		\Fitt_R(M) = \Fitt_R^0(M) \subseteq \Fitt_R^1(M)
		\subseteq \Fitt_R^2(M) \subseteq \dots
	\]
	of $R$-ideals. Moreover, if $M$ can be generated by $q$ elements,
	then $\Fitt_R^q(M) = R$.  See \cite[\S 3, Theorem 2]{MR0460383} 
	for a proof. If $R$ is a principal ideal domain, then the higher
	Fitting ideals $\Fitt^i_R(M)$ for all $i \geq 0$ determine 
	the $R$-module $M$ up to isomorphism. This can be deduced from the
	structure theorem of finitely generated modules over principal ideal
	domains (see \cite[\S 1.1]{MR1998607}, for instance). 
\end{remark}

Let us denote the $R$-annihilator ideal of $M$ by $\Ann_R(M)$.
The main interest in Fitting ideals comes from the following fact.

\begin{theorem} \label{thm:fitt-ann-comm}
	Let $R$ be a commutative ring and let $M$ be a finitely presented 
	$R$-module. Then one has an inclusion
	\[
		\Fitt_R(M) \subseteq \Ann_R(M).
	\]
\end{theorem}

\begin{proof}
	Choose a finite presentation \eqref{eqn:finite-presentation-comm}
	of $M$.	Let $H \in S_b(h)$ be a submatrix. As $M$ is a
	homomorphic image of the cokernel of $H$, we may assume that
	$h = H$ and thus $a=b$. Let $H^{\ast} \in M_{b \times b}(R)$
	be the adjoint matrix of $H$. Then 
	$H^{\ast} H = H H^{\ast} = \det(H) \boldsymbol{1}_{b}$
	and so the result follows from the
	commutative diagram
	\[
	\xymatrix{
		R^{b} \ar@{>}[rr]^{H}  & & 
		R^{b} \ar@{>}[d]^{\det(H)} \ar@{>}[lld]_{H^{\ast}}  \ar@{>>}[rr] & & M \ar@{>}[d]^{\det(H)} \\
		R^{b} \ar@{>}[rr]^{H} & & R^{b}  \ar@{>>}[rr] & & M
	}
	\]
	once one notes that the right vertical arrow is zero.
\end{proof}

\begin{example} \label{ex:Fitting-Z}
	Let $R = \Z$ and let $M$ be an arbitrary finitely generated $\Z$-module.
	By the fundamental theorem of finitely generated abelian groups
	there are unique integers $f,n \geq 0$ and positive integers
	$1 < d_1 \mid d_2 \mid \dotsm\mid d_n$ such that
	\[
		M \simeq \Z^f \oplus \bigoplus_{i=1}^n \Z/ d_i \Z.
	\]
	If $f > 0$ then clearly $\Fitt_{\Z}(M) = 0$. If $f=0$
	we can take for $h$ the diagonal matrix with entries
	$d_1, \dots, d_n$. It follows that
	\[
		\Fitt_{\Z}(M) = \left\{ \begin{array}{lll}
		0 & \mathrm{ if } & f>0\\
		(\prod_{i=1}^{n} d_i) \Z & \mathrm{ if } & f=0.
		\end{array}\right.
	\]
	Moreover, we clearly have
	\[
	\Ann_{\Z}(M) = \left\{ \begin{array}{lll}
	0 & \mathrm{ if } & f>0\\
	d_n \Z & \mathrm{ if } & f=0.
	\end{array}\right.
	\]
	In particular, the inclusion $\Fitt_{\Z}(M) \subseteq \Ann_{\Z}(M)$
	is proper if and only if $f=0$ and $n>1$. Of course, similar observations
	hold for every principal ideal domain $R$.
\end{example}

\begin{example} \label{ex:augmentation-ideal}
	Let $G$ be a finite abelian group and let $\Delta G$ be the kernel of
	the natural augmentation map $\Z[G] \rightarrow \Z$
	that sends each $g \in G$ to $1$. It is straightforward to show that
	\[
		\Ann_{\Z[G]}(\Delta G) = N_G  \Z,
	\]
	where $N_G := \sum_{g \in G} g$ (see \cite[Satz 1.3]{MR3014997}). 
	By Theorem \ref{thm:fitt-ann-comm}
	we must have
	\[
	\Fitt_{\Z[G]}(\Delta G) = m N_G \Z
	\]
	for some integer $m$. We now apply Remark \ref{rem:base-change}
	with $R = \Z[G]$ and $S = \Z$ so that
	\[
		\Fitt_{\Z}(\Z \otimes_{\Z[G]} \Delta G) = m |G| \Z.
	\]
	Moreover, we have isomorphisms of abelian groups
	$\Z \otimes_{\Z[G]} \Delta G
	\simeq \Delta G / (\Delta G)^2 \simeq G$ so that
	\[
	\Fitt_{\Z}(\Z \otimes_{\Z[G]} \Delta G) = |G| \Z
	\]
	by Example \ref{ex:Fitting-Z}. It follows that $m \in \Z^{\times}$
	and thus
	\[
	\Fitt_{\Z[G]}(\Delta G) = N_G \Z = \Ann_{\Z[G]}(\Delta G).
	\]
\end{example}

We now record some basic facts about Fitting ideals.

\begin{lemma} \label{lem:basic-props-comm}
	Let $R$ be a commutative ring and let $M_1$, $M_2$, $M_3$
	be finitely presented $R$-modules. 
	\begin{enumerate}
		\item 
		If $\pi: M_1 \twoheadrightarrow M_2$ is an epimorphism, then
		$\Fitt_R(M_1) \subseteq \Fitt_R(M_2)$.
		\item
		Fitting ideals behave well under direct sums:
		\[
			\Fitt_R(M_1 \oplus M_3) = \Fitt_R(M_1) \cdot \Fitt_R(M_3).
		\]
		\item
		If $M_1 \stackrel{\iota}{\rightarrow} M_2 \rightarrow M_3 \rightarrow 0$
		is an exact sequence, then
		\[
			\Fitt_R(M_1) \cdot \Fitt_R(M_3) \subseteq \Fitt_R(M_2).
		\]
	\end{enumerate}
\end{lemma}

\begin{proof}
	 We only sketch the proof. Let $R^{a_1} \stackrel{h_1}{\longrightarrow}
	 R^{b_1} \stackrel{\pi_1}{\longrightarrow} M_1 \rightarrow 0$ be a 
	 finite presentation of $M_1$. 
	 Put $\pi_2 := \pi \circ \pi_1$.
	 Then one may construct a finite presentation 
	 \[
		 R^{a_2} \xrightarrow{(h_1 \mid \ast)}
		 R^{b_1} \stackrel{\pi_2}{\longrightarrow} M_2 \longrightarrow 0
	 \]
	 of $M_2$ by adding more relations if necessary. This shows (i).
	 For (iii) we may therefore assume that $\iota$ is injective.
	 Let $h_1$ and $h_3$ be finite presentations of $M_1$ and $M_3$,
	 respectively. As in the proof of the horseshoe lemma (see
	 \cite[Lemma 2.2.8]{MR1269324}, for instance) one can construct
	 a finite presentation $h_2$ of $M_2$ of shape
	 \[
		 \left( \begin{array}{cc}
		 h_1 & g \\ 0 & h_3
		 \end{array}
		 \right).
	 \]
	 If $M_2 = M_1 \oplus M_3$ one may additionally assume that $g=0$.
	 From this one can deduce (ii) and (iii).
\end{proof}

\begin{example}
	Let $I_1, \dots, I_n$ be finitely generated ideals of $R$.
	Then it follows from Example \ref{ex:R/I} and Lemma
	\ref{lem:basic-props-comm}(ii) that
	\[
		\Fitt_R\left(\bigoplus_{j=1}^n R/I_j\right) = \prod_{j=1}^{n} I_j.
	\]
\end{example}

\begin{example}
	Let $p$ be a prime and let $R = \Z_p\llbracket T \rrbracket$ be the
	power series ring in one variable over $\Z_p$. 
	Let $M$ be a finitely generated
	torsion $R$-module. Then by the structure theorem for Iwasawa modules
	\cite[Theorem 5.3.8]{MR2392026} there is a pseudo-isomorphism
	\[
		\alpha: M \longrightarrow \bigoplus_{i = 1}^s R / p^{m_i} \oplus
		\bigoplus_{j = 1}^t R / F_j^{n_j},
	\]
	where $s,t \geq 0$, $m_i, n_j \geq 1$ are integers and the $F_j$
	are distinguished irreducible polynomials. 
	This means in particular that $\alpha$ becomes an isomorphism in every localisation
	of $R$ at a prime ideal of height $1$.
	The characteristic ideal
	$\Char_R(M)$ of $M$ is defined to be the $R$-ideal generated by
	$\prod_{i=1}^s p^{m_i} \cdot \prod_{j=1}^t F_j^{n_j}$.
	
	Now assume that $M$ contains no finite non-trivial submodule
	(that is $\alpha$ is injective). Then the projective dimension
	of $M$ is at most $1$ by 
	\cite[Proposition 5.3.19(i)]{MR2392026}. 
	Since $R$ is a local ring, every projective
	$R$-module is free and so there is a short exact sequence
	\[
		0 \longrightarrow R^a \longrightarrow R^a \longrightarrow
		M \longrightarrow 0.
	\]
	It follows that $\Fitt_R(M)$ is a principal ideal. Since two
	principal ideals over $R$ are equal if and only if they become
	equal in every localisation of $R$ at a height $1$ prime ideal,
	we have that
	\[
		\Fitt_R(M) = \Char_R(M)
	\]
	in this case (see \cite[Lemma 9.1]{MR2908781}, for instance).
\end{example}

\section{Noncommutative Fitting invariants: Basic properties}

\subsection{Fitting domains and Fitting orders}
We now introduce the class of rings for which we intend to define
Fitting invariants. We recall that an $F$-algebra $A$ over a field $F$
is called \emph{separable} if
$E \otimes_F A$ is a semisimple $E$-algebra for every field extension
$E$ of $F$. If $F$ is a perfect field, then every finite dimensional
semisimple $F$-algebra is indeed separable 
(as follows from \cite[Corollary 7.6]{MR632548}).

\begin{definition}
Let $\mathfrak o$ be an integrally closed commutative
noetherian complete local domain with field of quotients $F$.
Then we call $\mathfrak o$ a \emph{Fitting domain}.
Let $A$ be a finite dimensional separable $F$-algebra and let $\Lambda$
be an $\mathfrak o$-order in $A$. Then $\Lambda$ is called a
\emph{Fitting order} over $\mathfrak o$.
\end{definition}

\begin{remark}
	Let $\Lambda$ be a Fitting order over the Fitting domain $\mathfrak o$. 
	Then $\Lambda$ is noetherian and so every 
	finitely generated $\Lambda$-module 
	is in fact finitely presented.
\end{remark}

\begin{example}
	Any complete discrete valuation ring $\mathfrak o$ is a Fitting domain.
	Conversely, any Fitting domain of Krull dimension $1$ is a complete
	discrete valuation ring.
\end{example}

\begin{example}
	For any Fitting domain $\mathfrak o$ and any positive integer $n$,
	the ring $M_{n \times n}(\mathfrak o)$ is a Fitting order over
	$\mathfrak o$.
\end{example}

\begin{example}
	Let $p$ be a prime and let $G$ be a finite group.
	Then the ring of $p$-adic integers $\Z_p$ is a Fitting domain
	and the group ring $\Z_p[G]$ is a Fitting order over $\Z_p$.
\end{example}

\begin{example}
	More generally, let $\mathfrak o$ be an arbitrary Fitting domain
	with field of quotients $F$ and let $G$ be a finite group. Then
	an $\mathfrak o$-order $\Lambda$ in $A := F[G]$ is a Fitting order
	over $\mathfrak o$ if and only if $|G|$ is invertible in $F$.
\end{example}

\begin{example} \label{ex:Iwasawa-algebra}
	Let $G$ be a profinite group containing a finite normal subgroup $H$
	such that $G/H \simeq \Gamma$, where $\Gamma$ is a pro-$p$ group
	isomorphic to $\Z_p$. 
	Note that $G$ can be written as a semi-direct product $H \rtimes \Gamma$
	and is a one-dimensional $p$-adic Lie group.
	The Iwasawa algebra of $G$ over $\Z_p$
	is defined to be
	\[
		\Z_{p}\llbracket G\rrbracket = \varprojlim \Z_{p}[G/N],
	\]
	where the inverse limit is taken over all open normal subgroups 
	$N$ of $G$. Since any homomorphism $\Gamma \rightarrow \Aut(H)$
	must have open kernel, we may choose
	a natural number $n$ such that $\Gamma^{p^n}$ is central
	in $G$. We put 
	$\mathfrak o := \Z_p \llbracket \Gamma^{p^n} \rrbracket$
	and $F := \Quot(\mathfrak o)$.
	Then $\mathfrak o$
	is non-canonically isomorphic to the power series ring
	$\Z_p \llbracket T \rrbracket$ in one variable over $\Z_p$
	and is thus a Fitting domain. If we view $\Z_{p}\llbracket G\rrbracket$
	as an $\mathfrak o$-module (or indeed as a left $\mathfrak o[H]$-module),
	there is a decomposition
	\[
		\Z_p \llbracket G \rrbracket =
		\bigoplus_{i=0}^{p^n-1} \mathfrak o[H] \gamma^i,
	\]
	where $\gamma$ is a topological generator of $\Gamma$.
	This shows that $\Z_p \llbracket G \rrbracket$ is a Fitting
	order over $\mathfrak o$ in the separable 
	$F$-algebra 
	$A = \mathcal{Q}(G) := \oplus_i F[H]\gamma^i$.
\end{example}

\subsection{Reduced norms and the integrality ring} \label{subsec:reduced-norms}

Let $\mathfrak o$ be a Fitting domain with field of quotients $F$
and let $A$ be a finite dimensional separable $F$-algebra.
By Wedderburn's theorem $A$ decomposes into
\[
	 A = A_1 \oplus \dots \oplus A_t,
\]
where each $A_i$ is isomorphic to an algebra of $n_i \times n_i$
matrices over a skewfield $D_i$.
Then $F_i := \zeta(A_i) = \zeta(D_i)$ is a finite field extension
of $F$ and $A_i$ is a central simple $F_i$-algebra.
The reduced norm map
\[
	\Nrd = \Nrd_A: A \longrightarrow \zeta(A) = F_1 \oplus \dots \oplus F_t
\] 
is defined componentwise and extends to matrix rings
over $A$ in the obvious way (see \cite[\S 7D]{MR632548}).
If every $D_i$ is in fact a field,
then the reduced norm of $x =(x_i)_i \in A$ is indeed given by
$\Nrd(x) = (\det(x_i))_i$. In general, one can always choose
a (finite) field extension $E$ of $F$ such that $A_E := E \otimes_F A$
is of this form. Then one puts $\Nrd_A(x) := \Nrd_{A_E}(1 \otimes x)$
which actually belongs to $\zeta(A)$ and is independent of the choice of $E$.

Now let $\Lambda$ be a Fitting order in $A$ over $\mathfrak o$.
By \cite[Corollary 10.4]{MR1972204} we may choose a maximal 
$\mathfrak o$-order $\Lambda'$ in $A$ containing $\Lambda$.
Then $\Lambda'$ is also a Fitting order over $\mathfrak o$ 
and likewise decomposes into $\Lambda' = \Lambda_1' \oplus \dots
\oplus \Lambda_t'$, where $\Lambda_i'$ is a maximal $\mathfrak o$-order
in $A_i$ for each $i$.
The reduced norm restricts to a map
\[
	\Nrd: \Lambda' \longrightarrow \zeta(\Lambda') = 
	\mathfrak o_1 \oplus \dots \oplus \mathfrak o_t,
\]
where $\mathfrak o_i = \zeta(\Lambda_i')$ 
denotes the integral closure of $\mathfrak o$
in $F_i$. Unfortunately, it is in general not true that the reduced norm
maps $\Lambda$ into its centre.

\begin{example} \label{ex:dihedral-denominators}
	Let $p$ be an odd prime and let $D_{2p}$ be
	the dihedral group of order $2p$.
	We may write
	\[
		D_{2p} = \langle \sigma, \tau \mid \sigma^p = \tau^2 = 1, \tau \sigma
		= \sigma^{-1} \tau \rangle.
	\]
	Then $\Lambda := \Z_p[D_{2p}]$ is a Fitting order
	in $A := \Q_p[D_{2p}]$ over $\Z_p$
	and we wish to compute $\Nrd(\sigma + \tau)$.
	We put $E := \Q_p(\zeta_p)$, where $\zeta_p$ denotes a primitive
	$p$-th root of unity, and let $j \in \mathrm{Gal}(E/\Q_p)$
	be the unique automorphism of order $2$.
	By \cite[Example 7.39]{MR632548} we have the
	Wedderburn decomposition
	\begin{equation} \label{eqn:D2p-Wedderburn}
		A \simeq A_1 \oplus A_2 \oplus A_3,
	\end{equation}
	where $A_1 = A_2 = \Q_p$ and $A_3$ is the twisted group algebra
	$E \oplus E y$ with relations $y^2 = 1$ and 
	$y \alpha = j(\alpha) y$, $\alpha \in E$. 
	Moreover, for $\alpha + \beta y \in A_3$ one has
	\[
		\Nrd(\alpha + \beta y) = 
		N_{E^+/\Q_p}(\alpha j(\alpha) - \beta j(\beta)),
	\]
	where $E^+$ denotes the fixed field of $E$ 
	under the action of $j$ and $N_{E^+/\Q_p}: E^+
	\rightarrow \Q_p$ is the field-theoretic norm map.
	The isomorphism \eqref{eqn:D2p-Wedderburn} maps $\sigma$ to the triple
	$(1,1,\zeta_p)$ and $\tau$ to the triple $(1,-1,y)$.
	As the first factor in this decomposition corresponds to the central
	idempotent $e_1 := \frac{1}{2p} \sum_{\delta \in D_{2p}} \delta$,
	we have
	\[
		\Nrd(\sigma + \tau) = 2 e_1 = 
		\frac{1}{p} \sum_{\delta \in D_{2p}} \delta
		\not\in \zeta(\Lambda).
	\]
\end{example}

To overcome this problem we define a $\zeta(\Lambda)$-submodule 
of $\zeta(A)$ by
\[
	\mathcal{I}(\Lambda) := \langle \Nrd(H) \mid H \in 
	M_{b \times b}(\Lambda), b \in \N \rangle_{\zeta(\Lambda)}.
\]
Note that this is in fact a commutative $\mathfrak o$-order in $\zeta(A)$
contained in $\zeta(\Lambda')$. We call $\mathcal{I}(\Lambda)$ the
\emph{integrality ring} of $\Lambda$.
This is the smallest ring that contains $\zeta(\Lambda)$ and the image
of the reduced norm of all matrices with entries in $\Lambda$.

\begin{example} \label{ex:dihedral-int-ring}
	Let $\ell$ and $p$ be primes with $p$ odd. Choose a maximal
	$\Z_{\ell}$-order $\mathfrak M_{\ell}(D_{2p})$ containing
	$\Z_{\ell}[D_{2p}]$. Then one has (see \cite[Example 6]{MR3092262}
	and \cite[Proposition 6.9]{MR3461042})
	\[
		\mathcal I(\Z_{\ell}[D_{2p}]) = \left\{
		\begin{array}{lll}
		\zeta(\mathfrak M_{p}(D_{2p})) & \mathrm{ if } & \ell = p\\
		\zeta(\Z_{\ell}[D_{2p}]) & \mathrm{ if } & \ell \not= p.
		\end{array}
		\right.
	\]
	Note that the case $\ell \not= p$ follows from Proposition
	\ref{prop:best-denominators} below. For the case $\ell = p$
	one has to compute the reduced norms of several
	group ring elements as in Example \ref{ex:dihedral-denominators}
	and then show that these generate $\zeta(\mathfrak M_{p}(D_{2p}))$
	as a $\Z_p$-module.
\end{example}

\begin{remark}
	The integrality ring appears in many conjectures on the integrality
	of so-called Stickelberger elements. These elements lie in the centre
	of the rational group ring and are constructed via integer values of
	Artin $L$-functions attached to the irreducible characters of $G$,
	where $G$ is the Galois group of a finite Galois extension of 
	number fields. We refer the reader to 
	\cite{Brumer-Gross-Stark} for a survey of conjectures 
	and results in this context.
\end{remark}

\subsection{Noncommutative Fitting invariants}
Let $\Lambda$ be a Fitting order over the Fitting domain $\mathfrak o$.
Let $M$ be a $\Lambda$-module with finite presentation
\[
	\Lambda^a \stackrel{h}{\longrightarrow} \Lambda^b \longrightarrow
	M \longrightarrow 0.
\]
As before we let $S_b(h)$ be the set of all $b \times b$ submatrices
of $h$. Since the reduced norm is a generalisation of
the determinant with values in $\mathcal{I}(\Lambda)$, 
it is now natural to make the following definition.
\[
	\Fitt_{\Lambda}(h) := \left\{\begin{array}{lll}
	0 & \mathrm{ if } & a < b\\
	\langle \Nrd(H) \mid H \in S_b(h) \rangle_{\mathcal{I}(\Lambda)}
	& \mathrm{ if } & a \geq b.
	\end{array}\right.
\]
Unfortunately, this definition depends on $h$.

\begin{example} \label{ex:dependence-on-h}
	Consider the Fitting order $\Lambda = M_{2 \times 2}(\Z_3)$
	over $\Z_3$ and the trivial $\Lambda$-module $M=0$.
	We have $\mathcal{I}(\Lambda) = \zeta(\Lambda) = \Z_3$.
	The identity map $\id: \Lambda \rightarrow \Lambda$
	is certainly a finite presentation of $M$ and we have
	$\Fitt_{\Lambda}(\id) = \langle \Nrd(\id) \rangle_{\Z_3} = \Z_3$.
	However, the map
	\begin{eqnarray*}
	h: \Lambda e_1 \oplus \Lambda e_2 & \longrightarrow & \Lambda \\
	e_1 & \mapsto & \left(\begin{array}{cc} 4 & 1 \\ 1 & 4
	\end{array}\right) \\
	e_2 & \mapsto & \left(\begin{array}{cc} 5 & 1 \\ 1 & 5
	\end{array}\right)
	\end{eqnarray*}
	is also a finite presentation of $M$ and we have
	$\Fitt_{\Lambda}(h) = \langle 15, 24 \rangle_{\Z_3} = 3 \Z_3$.
\end{example}

\begin{remark}
	The ring $\Lambda$ in Example \ref{ex:dependence-on-h}
	is a matrix ring over a commutative ring.
	In this case one can remedy the dependence on $h$ via 
	Morita equivalence. We will explain this approach in
	the appendix.
\end{remark}

In order to examine the dependence on $h$ we try to adapt the proof
of Theorem \ref{thm:Fitt-well-comm} in the commutative case.
We still may view a finite presentation of $M$ as a truncated free
resolution of $M$. As $\Lambda$ is a semiperfect ring
(see \cite[Example 23.3]{MR1838439}),
every finitely generated module $M$ has a projective cover
(see \cite[Theorem 6.23]{MR632548} or \cite[Theorem 24.16]{MR1838439}):
there is a finitely generated projective module $P_0$ (unique up to
isomorphism) and a surjective map $\pi: P_0 \twoheadrightarrow M$
such that no proper submodule of $P_0$ is mapped onto $M$ by $\pi$.
If every projective $\Lambda$-module is free, then $P_0 = \Lambda^{b'}$
with $b' \in \N$ minimal such that there is a surjection 
$\Lambda^{b'} \twoheadrightarrow M$. Let $P_1$ be a projective cover
of the kernel of $\pi$. We obtain an exact sequence
\[
	P_1 \longrightarrow P_0 \longrightarrow M \longrightarrow 0
\]
which we may view as the truncation of a `minimal projective resolution'
$\mathcal{P}_M$ of $M$. 
This is the correct analogue of a minimal free resolution
of a module over a commutative local ring:
Any free (in fact any projective) resolution of $M$ is isomorphic to
the direct sum of $\mathcal{P}_M$ and a trivial complex
\cite[Proposition 2.1]{MR2609173}.
Now let
\[
\Lambda^{a'} \stackrel{h'}{\longrightarrow} \Lambda^{b'} \longrightarrow
M \longrightarrow 0
\]
be a second finite presentation of $M$. By similar arguments as in
the commutative case one may assume that $a=a'$, $b=b'$ and
that there are matrices $X \in \GL_a(\Lambda)$ and $Y \in \GL_b(\Lambda)$
such that
\[
	h \circ X = Y \circ h'.
\]
As $\Nrd(Y)$ belongs to $\mathcal{I}(\Lambda)^{\times}$, we may
assume in addition that $Y=1$. In contrast to the determinant,
the reduced norm is not a multilinear map so that we cannot assume
that $X=1$. However, assuming $h \circ X = h'$ as we may, we can construct
a new finite presentation of $M$, namely
\[
\Lambda^a \oplus \Lambda^a \xrightarrow{(h \mid h')}
\Lambda^b \longrightarrow
M \longrightarrow 0.
\]
Now $\Fitt_{\Lambda}((h \mid h'))$ contains both $\Fitt_{\Lambda}(h)$
and $\Fitt_{\Lambda}(h')$. As $\mathcal{I}(\Lambda)$ is
a noetherian ring, we have shown the following
(see also \cite[Theorem 3.2 and Definition 3.3]{MR2609173} and
\cite[\S 3.5]{MR3092262}).

\begin{theorem} \label{thm:max-Fitting}
	Let $\Lambda$ be a Fitting order and let $M$ be a finitely generated
	$\Lambda$-module. Then there is a finite presentation $h$ of $M$
	such that $\Fitt_{\Lambda}(h)$ contains $\Fitt_{\Lambda}(h')$
	for every other choice $h'$ of finite presentation of $M$. 
\end{theorem}

\begin{definition}
	Using the notation of Theorem \ref{thm:max-Fitting}, we put
	\[
	\Fitt_{\Lambda}^{\max}(M) := \Fitt_{\Lambda}(h)
	\]
	and call this the \emph{maximal Fitting invariant} of $M$ over $\Lambda$.
\end{definition}

\begin{remark}
	For an axiomatic approach to 
	noncommutative Fitting invariants we refer the reader to \cite{Kataoka}.
	A natural notion of `higher noncommutative Fitting invariants' 
	has recently been considered
	by Burns and Sano \cite{non-abelian-zeta}.
\end{remark}

\begin{remark}
	In order to define noncommutative Fitting invariants it suffices
	to assume that $\mathfrak o$ is a commutative
	noetherian complete local domain. In fact, the integral closure
	of $\mathfrak o$ in its field of quotients is finitely generated
	as an $\mathfrak o$-module by \cite[Theorem 4.3.4]{MR2266432} 
	in this case, and noncommutative Fitting invariants have been
	defined in this greater generality in \cite{MR2609173}.
\end{remark}

One can prove the analogues of Lemma \ref{lem:basic-props-comm}(i)
and (iii) without any significant changes.

\begin{lemma} \label{lem:basic-props-noncomm}
	Let $\Lambda$ be a Fitting order and let $M_1$, $M_2$, $M_3$
	be finitely generated $\Lambda$-modules. 
	\begin{enumerate}
		\item 
		If $\pi: M_1 \twoheadrightarrow M_2$ is an epimorphism, then
		$\Fitt_{\Lambda}^{\max}(M_1) \subseteq \Fitt_{\Lambda}^{\max}(M_2)$.
		\item
		If $M_1 \rightarrow M_2 \rightarrow M_3 \rightarrow 0$
		is an exact sequence, then
		\[
		\Fitt_{\Lambda}^{\max}(M_1) \cdot \Fitt_{\Lambda}^{\max}(M_3) 
		\subseteq \Fitt_{\Lambda}^{\max}(M_2).
		\]
	\end{enumerate}
\end{lemma}

However, the proof of Lemma \ref{lem:basic-props-comm}(ii) only gives
an inclusion
\[
	\Fitt_{\Lambda}^{\max}(M_1) \cdot \Fitt_{\Lambda}^{\max}(M_3)
	\subseteq \Fitt_{\Lambda}^{\max}(M_1 \oplus M_3)
\]
which is a special case of Lemma \ref{lem:basic-props-noncomm}(ii).
We will treat the question of whether this inclusion is actually an equality
in \S \ref{sec:additivity} below.\\

It is hard to decide in general whether a given presentation gives
a maximal Fitting invariant. In this direction we have the following
result.

\begin{proposition} \label{prop:Fitt-of-quadratic}
	Let $\Lambda$ be a Fitting order and let $M$ be a finitely generated
	$\Lambda$-module. If $M$ admits a quadratic presentation $h$, i.e.~a
	finite presentation of the form
	\[
	\Lambda^a \stackrel{h}{\longrightarrow} \Lambda^a \longrightarrow
	M \longrightarrow 0,
	\]
	then $\Fitt_{\Lambda}^{\max}(M) = \Fitt_{\Lambda}(h)$.
\end{proposition}

\begin{proof}
 This follows from \cite[Proposition 1.1 (4)]{MR2976321}.	
\end{proof}

\begin{remark}
	We briefly discuss the relation of noncommutative Fitting invariants
	to algebraic $K$-theory. 
	This will not be used in the following.
	For background on algebraic $K$-theory we refer the reader to
	\cite{MR892316} and \cite{MR0245634}.
	
	Suppose that $M$ admits a quadratic presentation
	$h$ which is injective. Then $M$ is torsion as an $\mathfrak o$-module
	and therefore defines a class $[M]$ in the relative algebraic $K$-group
	$K_0(\Lambda, A)$ associated to the ring homomorphism 
	$\Lambda \hookrightarrow A$. Moreover, the matrix $h$ belongs
	to $\GL_a(A)$ and so defines a class $[h]$ in $K_1(A)$ such that
	$\partial([h]) = [M]$, where $\partial: K_1(A) \rightarrow
	K_0(\Lambda,A)$ denotes the connecting homomorphism of relative
	algebraic $K$-theory. The reduced norm induces a group homomorphism
	$\Nrd: K_1(A) \rightarrow \zeta(A)^{\times}$ such that 
	$\Nrd([h]) = \Nrd(h)$. Now suppose that $x \in K_1(A)$ is a second
	pre-image of $[M]$. Then $[h]x^{-1}$ lies in the image of $K_1(\Lambda)$
	and therefore $\Nrd(x) = \Nrd([h]) \cdot \Nrd(y)$ for some
	$y \in K_1(\Lambda)$. 
	As $\Nrd(y) \in \mathcal{I}(\Lambda)^{\times}$,
	the $\mathcal{I}(\Lambda)$-ideals generated by $\Nrd([h])$ 
	and $\Nrd(x)$ coincide.
	In other words, for any $x \in K_1(A)$ such that $\partial(x) = [M]$
	one has
	\[
		\Fitt_{\Lambda}^{\max}(M) = 
		\langle \Nrd(x) \rangle_{\mathcal{I}(\Lambda)}.
	\]
	Now suppose that $\Fitt_{\Lambda}^{\max}(M)$ is generated by some
	$\xi \in \zeta(A)^{\times}$. Then $\xi \Nrd(x)^{-1}$ belongs to
	$\mathcal{I}(\Lambda)^{\times}$, but we cannot conclude in general
	that $\xi \Nrd(x)^{-1}$ lies in $\Nrd(K_1(\Lambda))$.
	The more involved notion of $\Nrd(\Lambda)$-equivalence classes
	in \cite{MR2609173} is designed in such a way that this conclusion
	works.
\end{remark}

\section{Fitting invariants and annihilation}

\subsection{Generalised adjoint matrices}
If $R$ is a commutative ring and $M$ is a finitely presented $R$-module, 
we know by Theorem \ref{thm:fitt-ann-comm} that $\Fitt_R(M)$ is
always contained in the $R$-annihilator ideal of $M$. The main ingredient
of the proof was the existence of adjoint matrices. We now generalise 
this concept.

Let $\Lambda$ be a Fitting order.
Choose $n \in \N$ and let $H \in M_{n \times n}(\Lambda)$.
Then recalling the notation of \S \ref{subsec:reduced-norms},
decompose $H$ into
\[
H = \sum_{i=1}^{t} H_{i} \in M_{n \times n}(\Lambda') = \bigoplus_{i=1}^t  M_{n \times n}(\Lambda'_{i}).
\]
Let $m_{i} = n_{i} \cdot s_{i} \cdot n$, where $s_i$ denotes the Schur index
of $D_i$ so that $[D_i:F_i] = s_i^2$.
The reduced characteristic polynomial $f_{i}(X) = \sum_{j=0}^{m_{i}} \alpha_{ij}X^{j}$ of $H_{i}$
has coefficients in $\mathfrak{o}_{i}$.
Moreover, the constant term $\alpha_{i0}$ is equal to 
$\Nrd(H_{i}) \cdot (-1)^{m_{i}}$.
We put
\[
H_{i}^{\ast} := (-1)^{m_{i}+1} \cdot \sum_{j=1}^{m_i} \alpha_{ij}H_{i}^{j-1}, \quad H^{\ast} := \sum_{i=1}^{t} H_{i}^{\ast}.
\]
We call $H^{\ast}$ the \emph{generalised adjoint matrix} of $H$.

\begin{lemma}\label{lem:ast}
	We have $H^{\ast} \in M_{n\times n} (\Lambda')$ and $H^{\ast} H = H H^{\ast} = \Nrd(H) \cdot \boldsymbol{1}_{n}$.
\end{lemma}

\begin{proof}
	The first assertion is clear by the above considerations.
	Since $f_{i}(H_{i}) = 0$, we find that
	\[
	H_{i}^{\ast} \cdot H_{i} = H_{i} \cdot H_{i}^{\ast}  = (-1)^{m_{i}+1} (-\alpha_{i0}) = \Nrd(H_{i}),
	\]
	as desired.
\end{proof}

\begin{remark}
	Note that the above definition of $H^{\ast}$ differs slightly from the definition in \cite[\S 4]{MR2609173}.
	Here we follow the treatment in \cite[\S 3.6]{MR3092262}.
\end{remark}

\begin{remark} \label{rem:reduced-adjoint}
	Let $E/F$ be a separable field extension such that $A_E := 
	E \otimes_F A$ splits. We may view $H$ as an element
	of $M_{n \times n}(A_E)$ which is a finite sum of matrix rings	
	over $E$. Then $H^{\ast} \in M_{n \times n}(A_E)$ is just the
	sum of the adjoint matrices in each component. As such it might have
	been more natural to call $H^{\ast}$ the `reduced adjoint matrix'
	of $H$.
\end{remark}

\begin{example} \label{ex:0-ast}
	Let $0 \in M_{n \times n}(R)$ where $R$ is a commutative ring.
	Then for the adjoint matrix $0^{\ast}$ we have
	$0^{\ast} = 1$ if $n = 1$ and $0^{\ast} = 0$ if $n>1$.
	Let $p$ be a prime and let $G$ be a finite group. Denote the
	commutator subgroup of $G$ by $G'$ and let $E/\Q_p$
	be a splitting field for $\Q_p[G]$. 
	Then Wedderburn's theorem for the algebra $E[G]$ implies that
	for $0 \in M_{1 \times 1}(\Z_p[G])$ we have
	\[
			0^{\ast} = \frac{1}{|G'|} \sum_{g\in G'} g.
	\]
\end{example}

\begin{example} \label{ex:H1ast}
	Let $\Lambda$ be a Fitting order and let $H \in M_{n \times n}(\Lambda)$.
	Then for every positive integer $m$ one has
	\[
		\left(\begin{array}{cc}
		 H & 0 \\ 0  & \boldsymbol{1}_m
		\end{array}
		\right)^{\ast} = 
		\left(\begin{array}{cc}
		H^{\ast} & 0 \\ 0  & \Nrd(H) \boldsymbol{1}_m
		\end{array}
		\right).
	\]
	In view of Remark \ref{rem:reduced-adjoint}, this follows from
	the respective statement for adjoint matrices over commutative rings
	(for a more detailed proof see \cite[Theorem 1.7.8(iii)]{watson_thesis}).
\end{example}

\subsection{Denominator ideals}

We define
\[
\mathcal{H}(\Lambda)  := \{ x \in \zeta(\Lambda) \mid xH^{\ast} \in 
M_{b \times b}(\Lambda) \, \forall H \in M_{b \times b}(\Lambda) \, 
\forall b \in \N \}
\]
and call $\mathcal{H}(\Lambda)$ the \emph{denominator ideal} of $\Lambda$.
We claim that 
\begin{equation} \label{eqn:HI_equals_H}
\mathcal{H}(\Lambda) \cdot \mathcal{I}(\Lambda) = \mathcal{H}(\Lambda) 
\subseteq \zeta(\Lambda)
\end{equation}
and so $\mathcal{H}(\Lambda)$ is in fact an ideal in the 
$\mathfrak{o}$-order $\mathcal{I}(\Lambda)$.
We follow an argument of David Watson 
\cite[Lemma 1.10.9]{watson_thesis}.
Let $x \in \mathcal{H}(\Lambda)$ and
$H_1 \in  M_{b_1 \times b_1}(\Lambda)$, 
$H_2 \in  M_{b_2 \times b_2}(\Lambda)$ with positive integers $b_1$
and $b_2$. We have to show that $x \Nrd(H_1)H_2^{\ast}$ belongs
to $M_{b_2 \times b_2}(\Lambda)$. By Example \ref{ex:H1ast}
we may assume that $b_1 = b_2$. We now compute
\[
	x \Nrd(H_1)H_2^{\ast} = x H_1 H_1^{\ast} H_2^{\ast}
	= H_1 x (H_2 H_1)^{\ast} \in M_{b_2 \times b_2}(\Lambda).
\]

\begin{remark}
	The denominator ideal $\mathcal{H}(\Lambda)$ measures the failure
	of the generalised adjoint matrices to have entries in $\Lambda$.
\end{remark}

\begin{remark}
	Let $\Lambda'$ be a maximal order containing $\Lambda$. 
	The central conductor 
	of $\Lambda'$ over $\Lambda$
	is defined to be
	$\mathcal F(\Lambda) := \left\{x \in \zeta(\Lambda') \mid
	x \Lambda' \subseteq \Lambda \right\}$.
	It is clear from Lemma \ref{lem:ast} that we
	always have $\mathcal F(\Lambda) \subseteq \mathcal H(\Lambda)$.
	Note that in particular we have
	$\mathcal{H}(\Lambda') = \zeta(\Lambda')$ for every maximal
	Fitting order $\Lambda'$.
\end{remark}

We now consider the case of $p$-adic group rings in more detail.
If $p$ is a prime and $G$ is a finite group, we set
\[
	 \mathcal{I}_{p}(G) := \mathcal{I}(\Z_{p}[G]), \quad
	 \mathcal{H}_{p}(G)  :=  \mathcal{H}(\Z_{p}[G]).
\]

\begin{proposition} \label{prop:best-denominators}
	Let $p$ be prime and $G$ be a finite group. Then $\mathcal{H}_{p}(G) = \zeta(\Z_{p}[G])$ if and only if $p$ 
	does not divide the order of the commutator subgroup of $G$. Moreover, in this case we have
	$\mathcal{I}_{p}(G) = \zeta(\Z_{p}[G])$.
\end{proposition}

\begin{proof}
	The first claim is a special case of \cite[Proposition 4.4]{MR3092262}.
	Note that Example \ref{ex:0-ast} shows that
	$\mathcal{H}_{p}(G) = \zeta(\Z_{p}[G])$ is only possible if
	$p$ does not divide the order of the commutator subgroup of $G$.
	The second claim follows easily from \eqref{eqn:HI_equals_H}.
\end{proof}

Let $\overline{\Q}_p$ be a separable closure of $\Q_p$.
For an irreducible character 
$\chi: G \rightarrow \overline{\Q}_p$ we put
$\Q_p(\chi) := \Q_p(\chi(g) \mid g \in G)$.
In the case of $p$-adic group rings
the central conductor is explicitly given by 
Jacobinski's formula \cite{MR0204538}
(see \cite[Theorem 27.13]{MR632548})
\begin{equation} \label{eqn:conductor-formula}
\mathcal{F}_p(G) :=
\mathcal F(\Z_p[G]) = \bigoplus_{\chi} \frac{|G|}{\chi(1)}
\mathcal D^{-1} (\Q_p(\chi) / \Q_p),
\end{equation}
where $\mathcal D^{-1} (\Q_p(\chi)/ \Q_p)$ denotes the inverse different of
the extension $\Q_p(\chi)$ over $\Q_p$
and the sum runs over all irreducible characters of $G$ 
modulo the natural action of the absolute Galois group
of $\Q_p$ on the irreducible characters of $G$. 

\begin{example} \label{c7-ex:dihedral}
	Let $p$ and $\ell$ be primes with $p$ odd. 
	We consider the group ring $\Z_{\ell}[D_{2p}]$, 
	where $D_{2 p}$ denotes the dihedral group of order $2p$.
	In the case $p=3$, one has
	$D_{6} \simeq S_{3}$, the symmetric group on three letters.
%	We let $\mathfrak{M}_{\ell}(D_{2p})$ be a maximal $\Z_{\ell}$-order
%	containing $\Z_{\ell}[D_{2p}]$.
	Then we have
	\[
	\mathcal{H}_{\ell}(D_{2p}) = \left\{ \begin{array}{lll}
	\zeta(\Z_{\ell}[D_{2p}]) & \mbox{ if } & p\neq \ell \\
	\mathcal{F}_{p}(D_{2p}) & \mbox{ if } & p=\ell.
	\end{array}\right.
	\]
	In fact, the result follows from Proposition 
	\ref{prop:best-denominators} if $p \neq \ell$.
	In the case $p=\ell$, the result is established in 
	\cite[Example 6]{MR3092262}. The corresponding integrality rings
	have already been determined in Example \ref{ex:dihedral-int-ring}.
\end{example}

\begin{example} \label{ex:Aff(q)}
	Let $p$ be a prime and let $q = \ell^{n}$ be a prime power.
	We consider the group $\mathrm{Aff}(q) = \F_q \rtimes \F_q^{\times}$
	of affine transformations 
	on $\F_q$, the finite field 
	with $q$ elements.	
	Let $\mathfrak{M}_{p}(\Aff(q))$ be a maximal $\Z_{p}$-order such that $\Z_{p}[\Aff(q)] \subseteq \mathfrak{M}_{p}(\Aff(q)) \subseteq \Q_{p}[\Aff(q)]$.
	Then by \cite[Proposition 6.7]{MR3461042} we have
	\[
	\mathcal{H}_{p}(\Aff(q)) = \left\{ \begin{array}{lll}
	\zeta(\Z_{p}[\Aff(q)]) & \mbox{ if } & p \neq \ell \\
	\mathcal{F}_{p}(\Aff(q)) & \mbox{ if } & p=\ell \neq 2;
	\end{array}\right.
	\]
	\[
	\mathcal{I}_{p}(\Aff(q)) = \left\{ \begin{array}{lll}
	\zeta(\Z_{p}[\Aff(q)]) & \mbox{ if } & p \neq \ell\\
	\zeta(\mathfrak{M}_{p}(\Aff(q))) & \mbox{ if } & p=\ell \neq 2.
	\end{array}\right.
	\]
	If $p=\ell=2$, then we have containments
	\[
	2\mathcal{H}_{2}(\Aff(q)) \subseteq \mathcal{F}_{2}(\Aff(q)) \subseteq \mathcal{H}_{2}(\Aff(q)),
	\]
	\[
	2\zeta(\mathfrak{M}_{2}(\Aff(q))) \subseteq \mathcal{I}_{2}(\Aff(q)) \subseteq \zeta(\mathfrak{M}_{2}(\Aff(q))).
	\]
	Note that the commutator subgroup of $\Aff(q)$ is $\F_q$ so that
	the case $p \neq \ell$ again follows from Proposition
	\ref{prop:best-denominators}. An exact formula for the
	denominator ideal including the case
	$p = \ell = 2$ has been determined by David Watson 
	\cite[Example 3.6.6]{watson_thesis}
\end{example}

\begin{example}
	Let $S_4$ be the symmetric group on $4$ letters. 
	If $p$ is an odd prime, then 
	$\mathcal{I}_p(S_4) = \zeta(\mathfrak{M}_p(S_4))$
	and $\mathcal{H}_p(S_4) = \mathcal{F}_p(S_4)$. However, if $p=2$ we have
	\[
	\mathcal{F}_2(S_4) \subsetneq \mathcal{H}_2(S_4) 
	\subsetneq \zeta(\Z_2[S_4]);
	\]
	\[
	\zeta(\Z_2[S_4]) \subsetneq \mathcal{I}_2(S_4) 
	\subsetneq \zeta(\mathfrak M_2(S_4)).
	\]
	This follows from \cite[Proposition 6.8]{MR3461042}.
\end{example}

\begin{remark}
Even in the case of $p$-adic group rings, a general formula
for denominator ideals is still not available, though it would be
of significant interest for arithmetic applications. 
In particular, we seek good lower bounds. This question is extensively studied
in the PhD thesis of David Watson \cite{watson_thesis}. 
In particular, he
determines the denominator ideal $\mathcal{H}_p(G)$ for any
(non-abelian) group $G$ of order $p^3$.
\end{remark}

Now let $\Lambda$ be an arbitrary Fitting order and let $\Lambda'$
be a maximal order containing $\Lambda$. We define
a variant of the central conductor by
\[
	\mathcal{F}_{\zeta}(\Lambda) :=
	\left\{x \in \zeta(\Lambda') \mid
	x \zeta(\Lambda') \subseteq \zeta(\Lambda) \right\}.
\]
One clearly has an inclusion $\mathcal{F}(\Lambda)
\subseteq \mathcal{F}_{\zeta}(\Lambda)$, but this is not an equality
in general. 

\begin{example}
	Let $D_{2^a}$ be the dihedral group of order $2^a$, where
	$a \geq 3$.
	Then one can show (see \cite[Example 7]{MR3092262}) that
	\[
		[\mathcal{F}_{\zeta}(\Z_2[D_{2^a}]) : \mathcal{F}(\Z_2[D_{2^a}])]
		= 2^{a-2}.
	\]
\end{example}

\begin{remark}
	In the case where $\Lambda$ is a $p$-adic group ring one has an explicit
	formula for $\mathcal{F}_{\zeta}(\Lambda)$; see
	\cite[Proposition 6.12]{MR3092262}.
\end{remark}

Since the reduced characteristic polynomials have coefficients
in $\zeta(\Lambda')$, one can give the following lower bound
\cite[Proposition 6.3]{MR3092262} for $\mathcal{H}(\Lambda)$.

\begin{proposition}
	We have $\mathcal{F}_{\zeta}(\Lambda) \subseteq \mathcal{H}(\Lambda)$.
\end{proposition}

\subsection{Fitting invariants and annihilation}

Now a proof similar to the commutative case shows the desired
annihilation result (see \cite[Theorem 4.2]{MR2609173} 
and \cite[Theorem 3.3]{MR3092262}).

\begin{theorem}\label{thm:fitt-ann}
	Let $\Lambda$ be a Fitting order and let $M$ be a finitely generated $\Lambda$-module. Then one has an inclusion
	\[
	\mathcal{H}(\Lambda) \cdot \Fitt_{\Lambda}^{\max}(M) \subseteq \Ann_{\zeta(\Lambda)}(M).
	\]
\end{theorem}

Since $\mathcal{F}(\Lambda)$ is contained in $\mathcal{H}(\Lambda)$,
the above
inclusion also holds with $\mathcal{H}(\Lambda)$ replaced by
$\mathcal{F}(\Lambda)$. However, if one wishes to compute annihilators
using $\mathcal{F}(\Lambda)$ then the following result shows that it
suffices to compute the Fitting invariant over the maximal order
(see \cite[Corollary 6.5 and Theorem 6.7]{MR3092262}).

\begin{proposition} \label{prop:ann-over-max}
	Let $\Lambda$ be a Fitting order and let $M$ be a 
	finitely generated $\Lambda$-module. Choose a maximal order
	$\Lambda'$ containing $\Lambda$. Then
	\[
		\mathcal{F}(\Lambda) \cdot \Fitt_{\Lambda}^{\max}(M) \subseteq
		\mathcal{F}(\Lambda) \cdot \Fitt_{\Lambda'}^{\max}
		(\Lambda' \otimes_{\Lambda} M) \subseteq
	    \Ann_{\zeta(\Lambda)}(M).
	\]
\end{proposition}

\begin{example}
	We generalise Example \ref{ex:augmentation-ideal}.
	Let $p$ be a prime and let $G$ be a finite group.
	Let $\Delta_p G$ be the kernel of the natural augmentation map
	$\aug_p: \Z_p[G] \rightarrow \Z_p$ that sends each $g \in G$ to $1$.
	As $|G|$ belongs to $\mathcal{H}_p(G)$ we have by
	Theorem \ref{thm:fitt-ann} that
	\[
		|G| \cdot \Fitt_{\Z_p[G]}^{\max}(\Delta_p G) \subseteq
		\Ann_{\Z_p[G]}(\Delta_p G) = N_G \Z_p,
	\]
	where as before $N_G := \sum_{g \in G} g$. It follows that
	\[
		\Fitt_{\Z_p[G]}^{\max}(\Delta_p G) = m \frac{1}{|G|} N_G \Z_p
	\]
	for some $m \in \Z_p$. Let $h$ be a finite presentation of 
	$\Delta_p G$ such that $\Fitt_{\Z_p[G]}(h) = 
	\Fitt_{\Z_p[G]}^{\max}(\Delta_p G)$. Let $\aug_p(h)$ be the matrix
	obtained from $h$ by applying $\aug_p$ to each of the entries of $h$.
	Then $\aug_p(h)$ is a finite presentation of the $\Z_p$-module
	\[
		\Z_p \otimes_{\Z_p[G]} \Delta_p G \simeq
		\Delta_p G / (\Delta_p G)^2 \simeq \Z_p \otimes_{\Z} G/G',
	\]
	where $G'$ denotes the commutator subgroup of $G$. It follows
	from Example \ref{ex:Fitting-Z} that
	\[
	m \Z_p = \Fitt_{\Z_p}(\aug_p(h)) =  
	\Fitt_{\Z_p}(\Z_p \otimes_{\Z_p[G]} \Delta_p G) = |G/G'| \Z_p.
	\]
	This implies that we may choose $m = |G/G'|$ and thus
	\[
	\Fitt_{\Z_p[G]}^{\max}(\Delta_p G) = \frac{1}{|G'|} N_G \Z_p.
	\]
	As $N_{G'} := \sum_{g \in G'} g$ belongs to $\mathcal{H}_p(G)$
	by \cite[Corollary 6.14]{MR3092262} we find that
	\[
		\mathcal{H}_p(G) \cdot \Fitt_{\Z_p[G]}^{\max}(\Delta_p G)
		= N_G \Z_p = \Ann_{\Z_p[G]}(\Delta_p G).
	\]
	Let $\mathfrak{M}_p(G)$ be a maximal order containing $\Z_p[G]$.
	One can likewise show that
	\[
		\Fitt_{\mathfrak M_p(G)}^{\max}(\mathfrak{M}_p(G) 
		\otimes_{\Z_p[G]} \Delta_p G)
		= \frac{1}{|G'|} N_G \Z_p.
	\]
	Then Proposition \ref{prop:ann-over-max} implies
	the weaker result
	\[
		\mathcal{F}(\Z_p[G]) \cdot 
		\Fitt_{\mathfrak M_p(G)}^{\max}(\mathfrak{M}_p(G) 
		\otimes_{\Z_p[G]} \Delta_p G) = 
		\frac{|G|}{|G'|} N_G \Z_p
		\subseteq \Ann_{\Z_p[G]}(\Delta_p G).
	\]
\end{example}

\section{$p$-adic group rings}

In this section we fix a prime $p$ and a finite group $G$.
The $p$-adic group ring $\Z_p[G]$ is a Fitting order over $\Z_p$
which is of particular interest in number theory.
We put $\Lambda := \Z_p[G]$ and $A := \Q_p[G]$.

For any $\Lambda$-module $M$ we write $M^{\vee}$ for its
Pontryagin dual $\Hom_{\Z_p}(M, \Q_p/\Z_p)$ and $M^{\ast}$ for the
linear dual $\Hom_{\Z_p}(M, \Z_p)$, each endowed with the natural
contragredient action of $G$. We denote by
$^{\sharp}: A \rightarrow A$ the anti-involution
which maps each $g \in G$ to its inverse. 
If $h \in M_{a \times b}(A)$ is a matrix we let $h^{\sharp}$
be the matrix obtained from $h$ by applying $^{\sharp}$ to each of its
entries. Moreover, we let $h^T \in M_{b \times a}(A)$ be the
transpose of $h$. We note that there is an isomorphism 
$\Lambda^{\ast} \simeq \Lambda$, $f \mapsto \sum_{g \in G} f(g)g$.
Under this identification the $\Z_p$-dual of a map
$h \in M_{a \times b}(\Lambda)$ identifies with 
$h^{T, \sharp} \in M_{b \times a}(\Lambda)$.

Now let $C$ be a finite $\Lambda$-module of projective dimension 
at most $1$. Choose $n \in \N$ and a surjective map
$\Lambda^n \twoheadrightarrow C$ with kernel $P$. Note that $P$
is projective. As $C$ is finite, we have an isomorphism
$\Q_p \otimes_{\Z_p} P \simeq A^n$ of $A$-modules. Now
Swan's theorem \cite[Theorem (32.1)]{MR632548} 
implies that in fact $P \simeq \Lambda^n$.
In particular, we find that $C$ has a quadratic presentation 
\begin{equation} \label{eqn:quadratic-presentation}
	0 \longrightarrow \Lambda^n \stackrel{q}{\longrightarrow}
	\Lambda^n \longrightarrow C \longrightarrow 0.
\end{equation}
Moreover, the maximal Fitting invariant $\Fitt_{\Lambda}^{\max}(C)$
is generated by $\Nrd(q) \in \zeta(A)^{\times}$ by
Proposition \ref{prop:Fitt-of-quadratic}.

We also note that a $\Lambda$-module is of projective dimension
at most $1$ if and only if it is a cohomologically trivial
$\Lambda$-module by \cite[Theorem 9]{MR0219512}.
The following result is very useful in computing
Fitting invariants over $p$-adic group rings. 

\begin{proposition} \label{prop:sequence-group-rings}
	Let $\Lambda := \Z_p[G]$ where $p$ is a prime
	and $G$ is a finite group.
	\begin{enumerate}
		\item 
		Let $C$ be a finite $\Lambda$-module of projective dimension
		at most $1$. Let $c \in \zeta(A)^{\times}$ be a generator
		of $\Fitt_{\Lambda}^{\max}(C)$. Then the Pontryagin dual
		$C^{\vee}$ is also a finite $\Lambda$-module of projective
		dimension at most $1$ and $\Fitt_{\Lambda}^{\max}(C^{\vee})$
		is generated by $c^{\sharp}$.
		\item
		Suppose we are given an exact sequence of finite $\Lambda$-modules
		\[
			0 \longrightarrow M \longrightarrow C \longrightarrow C'
			\longrightarrow M' \longrightarrow 0,
		\]
		where $C$ and $C'$ are of projective dimension at most $1$. 
		Then we have an equality
		\[
			\Fitt_{\Lambda}^{\max}(M^{\vee})^{\sharp} \cdot 
			\Fitt_{\Lambda}^{\max}(C') = 
			\Fitt_{\Lambda}^{\max}(M') \cdot
			\Fitt_{\Lambda}^{\max}(C).
		\]
	\end{enumerate}
\end{proposition}

\begin{proof}
	This follows from \cite[Proposition 5.3]{MR2609173}.
	Here we will give a new proof of (i) which is much
	shorter and easier than the original proof. The argument
	is inspired by \cite[Lemma 6]{MR2046598} and
	recent work of Kataoka \cite[\S 4]{Kataoka}.
	
	As $C$ is finite, we have 
	$\Hom_{\Z_p}(C, \Z_p) = \Hom_{\Z_p}(C, \Q_p) = 0$.
	As $\Q_p$ is an injective $\Z_p$-module, we have
	$\Ext^1_{\Z_p}(C, \Q_p) = 0$ and thus
	the short exact sequence 
	$\Z_p \hookrightarrow \Q_p \twoheadrightarrow \Q_p / \Z_p$
	induces	an isomorphism
	\[
		C^{\vee} \simeq \Ext_{\Z_p}^1(C, \Z_p).
	\]
	Now choose a quadratic presentation $q$ of $C$ as in 
	\eqref{eqn:quadratic-presentation}. 
	We may assume that $c = \Nrd(q)$.
	Note that $\Ext_{\Z_p}^1(\Lambda, \Z_p)$ vanishes,
	since $\Lambda$ is a projective $\Z_p$-module.
	We apply $\Z_p$-duals to \eqref{eqn:quadratic-presentation}
	and obtain an exact sequence
	\[
		0 \longrightarrow \Lambda^n \xrightarrow{q^{T,\sharp}}
		\Lambda^n \longrightarrow \Ext_{\Z_p}^1(C, \Z_p) \longrightarrow 0.
	\]
	As $\Nrd(q^{T,\sharp}) = c^{\sharp}$ we are done.
\end{proof}

\begin{remark}
	Note that exact sequences of the type considered in Proposition
	\ref{prop:sequence-group-rings} naturally occur in
	the context of the equivariant Tamagawa number conjecture
	as formulated by Burns and Flach \cite{MR1884523}.
	This conjecture refines and generalises a very wide range of well
	known results and conjectures relating special values of
	$L$-functions to certain natural arithmetic invariants.
	It thereby vastly generalises the analytic class number formula 
	for number fields and the Birch and Swinnerton-Dyer 
	conjecture for elliptic	curves (see \cite{MR2088713} for a survey).
\end{remark}

\begin{remark}
	There is also an analogue of Proposition \ref{prop:sequence-group-rings}
	for Iwasawa modules \cite[Proposition 6.3]{MR2609173}
	and even for more general Fitting orders \cite[\S 4]{Kataoka}.
	This has applications in the context of main conjectures
	of equivariant Iwasawa theory.
\end{remark}

\section{Additivity of Fitting invariants} \label{sec:additivity}

Let $\Lambda$ be a Fitting order over the Fitting domain $\mathfrak o$.
Let $M$ and $N$ be two finitely generated $\Lambda$-modules.
As already observed in Lemma \ref{lem:basic-props-noncomm}(ii), 
one always has an inclusion
\[
	\Fitt_{\Lambda}^{\max}(M) \cdot \Fitt_{\Lambda}^{\max}(N)
	\subseteq \Fitt_{\Lambda}^{\max}(M \oplus N).
\]

\begin{definition}
	The Fitting order $\Lambda$ is called \emph{Fitting-additive}
	if
	\[
	\Fitt_{\Lambda}^{\max}(M) \cdot \Fitt_{\Lambda}^{\max}(N)
	= \Fitt_{\Lambda}^{\max}(M \oplus N)
	\]
	for all finitely generated $\Lambda$-modules $M$ and $N$.
\end{definition}

The following observation is clear by Lemma \ref{lem:basic-props-comm}(ii).

\begin{proposition}
	Every commutative Fitting order $\Lambda$ is Fitting-additive.
\end{proposition}

As reduced norms are defined componentwise, the following is also
immediate.

\begin{lemma} \label{lem:add-add}
	Let $\Lambda_1$ and $\Lambda_2$ be Fitting orders over the Fitting domain
	$\mathfrak o$. Then $\Lambda_1 \oplus \Lambda_2$ is Fitting-additive
	if and only if both $\Lambda_1$ and $\Lambda_2$ are Fitting-additive.
\end{lemma}

We record some cases where it is known that $\Lambda$ is Fitting-additive.

\begin{theorem} \label{thm:Fitting-additive}
	The Fitting order $\Lambda$ is Fitting-additive in each of the 
	following cases.
	\begin{enumerate}
		\item 
		$\Lambda$ is a direct product of matrix rings over commutative rings.
		\item
		$\Lambda$ is a maximal order and $\mathfrak o$ is
		a complete discrete valuation ring.
	\end{enumerate}
\end{theorem}

\begin{proof}
%	We will prove case (i) in the appendix. However, we point out
%	that the centre of a Fitting order $\Lambda$ is a product of local rings
%	and so by Lemma \ref{lem:add-add} we can assume 
%	without loss of generality that
%	$R := \zeta(\Lambda)$ is local. As Morita equivalent rings have 
%	isomorphic centres by \cite[Corollary 18.42]{MR1653294}, 
%	we see that $\Lambda$ is Morita equivalent to
%	the local ring $R$ and therefore $\Lambda \simeq M_{n \times n}(R)$
%	for some $n$ by \cite[Corollary 18.36]{MR1653294}. 
	This follows from \cite[Theorem 4.6(ii)]{MR3092262}.
	We will reprove part (i) in the appendix 
	(see Remark \ref{rmk:compatible-Fitt-defs} 
	and Lemma \ref{lem:basic-props-Morita}(ii)).
\end{proof}

\begin{corollary} \label{cor:p-adic-additive}
	Let $p$ be a prime and let $G$ be a finite group. Suppose that
	$p$ does not divide the order of the commutator subgroup of $G$.
	Then the $p$-adic group ring $\Z_p[G]$ is Fitting-additive.
\end{corollary}

\begin{proof}
	It follows from \cite[Corollary, p.~390]{MR704622}
	that $\Z_p[G]$ is a direct product of matrix rings over commutative rings
	in this case. Thus the result follows from 
	Theorem \ref{thm:Fitting-additive}(i).
\end{proof}

\begin{corollary}
	Let $G$ be a profinite group containing a finite normal subgroup $H$
	such that $G/H \simeq \Z_p$ for some prime $p$. Suppose that
	$p$ does not divide the order of the commutator subgroup of $G$
	(which is finite). Then the Iwasawa algebra 
	$\Z_p \llbracket G \rrbracket$ is Fitting-additive.
\end{corollary}

\begin{proof}
	The Iwasawa algebra $\Z_p \llbracket G \rrbracket$ is again a
	direct product of matrix rings over commutative rings
	in this case by \cite[Proposition 4.5]{MR3092262}.
\end{proof}

We now give an example of a Fitting order $\Lambda$ which is 
\emph{not} Fitting-additive.

\begin{example} \label{ex:hereditary}
	Let $p$ be a prime and consider the $\Z_p$-order
	\[
		\Lambda := \left\{ \left(\begin{array}{cc}
		a & b \\ c & d
		\end{array}\right) \in M_{2 \times 2}(\Z_p) \mid b \equiv 0
		\mod p
		\right\}.
	\]
	This is a Fitting order over $\Z_p$ and one has 
	$\mathcal{H}(\Lambda) = \mathcal{I}(\Lambda) = \zeta(\Lambda) = \Z_p$.
	We let $M$ and $N$ be $\Lambda$-modules which as sets are equal to 
	$\Z_p / p \Z_p$ and upon which $\Lambda$ acts as follows.
	Let $\lambda = \left(\begin{array}{cc} a & b \\ c & d \end{array}\right)
	\in \Lambda$. For every $x \in \Z_p$ we write 
	$\overline x$ for its image in $\Z_p / p \Z_p$.
	Then $\lambda \cdot m := \overline{a} m$ and
	$\lambda \cdot n := \overline{d} n$ for $m \in M$ and $n \in N$.
	Using $\overline b = 0$ it is easily checked that this defines 
	left $\Lambda$-module structures on $M$ and $N$,
	respectively. There is
	a short exact sequence
	\[
		0 \longrightarrow \Lambda \stackrel{h}{\longrightarrow}
		\Lambda \longrightarrow M \oplus N \longrightarrow 0,
	\]
	where $h$ is right multiplication by 
	$\left(\begin{array}{cc} 0 & p \\ 1 & 0 \end{array}\right)
	\in \Lambda$. As $h$ is a quadratic presentation,
	we have that
	\[
		\Fitt_{\Lambda}^{\max}(M \oplus N) = \Nrd(h) \cdot \Z_p = p \Z_p
	\]
	by Proposition \ref{prop:Fitt-of-quadratic}. As $M \oplus N$
	surjects onto $M$, the ideal generated by $p$ is contained in 
	$\Fitt_{\Lambda}^{\max}(M)$ by Lemma \ref{lem:basic-props-noncomm}(i).
	However, the maximal Fitting invariant $\Fitt_{\Lambda}^{\max}(M)$
	annihilates $M$ by Theorem \ref{thm:fitt-ann} and so 
	$\Fitt_{\Lambda}^{\max}(M)$ is properly contained in $\Z_p$ 
	as $M \not=0$.
	It follows that
	\[
		\Fitt_{\Lambda}^{\max}(M) = p \Z_p.
	\]
	Exactly the same reasoning applies for $N$ and therefore
	\[
	\Fitt_{\Lambda}^{\max}(N) = p \Z_p.
	\]
	Altogether we have that
	\[
	\Fitt_{\Lambda}^{\max}(M)  \cdot \Fitt_{\Lambda}^{\max}(N) = p^2 \Z_p
	\subsetneq p \Z_p = \Fitt_{\Lambda}^{\max}(M \oplus N)
	\]
	and thus $\Lambda$ is not Fitting-additive.
\end{example}

The order in Example \ref{ex:hereditary} is a hereditary, but
non-maximal $\Z_p$-order.
By the classification of hereditary orders
over complete discrete valuation rings \cite[Theorem 26.28]{MR632548}
it is clear that similar examples can be constructed for every
hereditary, non-maximal order over a
complete discrete valuation ring. 
Taking Theorem \ref{thm:Fitting-additive}(ii) into account,
we have established the following.

\begin{proposition} \label{prop:hereditary-not-add}
	Let $\Lambda$ be a Fitting order over a complete discrete
	valuation ring. Suppose that $\Lambda$ is hereditary.
	Then $\Lambda$ is Fitting-additive if and only if it is
	maximal.
\end{proposition}

\begin{remark} \label{rem:hereditary-max}
	A $p$-adic group ring $\Z_p[G]$ is hereditary if and only if 
	it is maximal.
	We give an indirect proof of this fact.
	Suppose that $\Z_p[G]$ is hereditary. Then $\mathcal{H}_p(G) = 
	\zeta(\Z_p[G])$ by \cite[Corollary 4.2]{MR3092262}
	(this follows easily from the definitions once one observes
	that the centre of a hereditary $\Z_p$-order is itself
	a maximal $\Z_p$-order).
	Now Proposition \ref{prop:best-denominators} implies
	that $p$ does not divide the order of the commutator subgroup of $G$.
	By Corollary \ref{cor:p-adic-additive} the group ring
	$\Z_p[G]$ is Fitting-additive and so by Proposition
	\ref{prop:hereditary-not-add} it is maximal.
\end{remark}	

Let $\Lambda$ be a Fitting order over the Fitting domain $\mathfrak o$.
In view of Example \ref{ex:hereditary} and 
Proposition \ref{prop:hereditary-not-add}
one may ask whether there are hereditary, non-maximal orders
when $\mathfrak o$ is not a complete discrete valuation ring.
We now show that this is indeed not the case.

\begin{proposition}
	Let $\Lambda$ be a Fitting order over the Fitting domain $\mathfrak o$.
	If $\Lambda$ is hereditary, then $\mathfrak o$ is a complete discrete valuation ring.
\end{proposition}

\begin{proof}
	Suppose that $\Lambda$ is a hereditary Fitting order over $\mathfrak o$
	in the separable $F$-algebra $A$, where as before $F$ denotes the quotient
	field of $\mathfrak o$. Let $A = A_1 \oplus \dots \oplus A_t$ be
	the Wedderburn decomposition of $A$ so that each $A_i$ is a central simple
	$F_i = \zeta(A_i)$-algebra. 
	Let $\mathfrak o_i$ be the integral closure of $\mathfrak o$ in $F_i$.
	Then we likewise have a decomposition
	\[
		\Lambda= \Lambda_1 \oplus \dots \oplus \Lambda_t,
	\]
	where each $\Lambda_i$ is a hereditary $\mathfrak{o}_i$-order
	by \cite[Proposition 2.2]{MR0151489}.
	Moreover, each $\mathfrak o_i$ is in fact a Dedekind domain by
	\cite[Theorem 2.6]{MR0151489} and thus has Krull dimension $1$.
	The Fitting domain $\mathfrak{o}$ then also has Krull dimension $1$
	by \cite[Proposition 4.15]{MR1322960}. However, a Fitting domain
	of Krull dimension $1$ is a complete discrete valuation ring.
\end{proof}

\begin{remark}
	In view of Corollary \ref{cor:p-adic-additive} and Remark
	\ref{rem:hereditary-max} one may ask the following question:
	Is every $p$-adic group ring Fitting-additive?
	Similarly, is the Iwasawa algebra $\Z_p \llbracket G \rrbracket$
	of a one-dimensional $p$-adic Lie group $G$ always Fitting-additive?
	In both cases one knows that 
	\[
	\Fitt_{\Lambda}^{\max}(M) \cdot \Fitt_{\Lambda}^{\max}(N)
	= \Fitt_{\Lambda}^{\max}(M \oplus N)
	\]
	whenever at least one of the two $\Lambda$-modules $M$ and $N$
	has projective dimension at most $1$. This follows
	from \cite[Proposition 2.11]{Kataoka}.
\end{remark}

\appendix

\section{Fitting invariants and Morita equivalence \\
by Henri Johnston and Andreas Nickel}

\subsection{Preliminaries on Morita equivalence}

Let $\Lambda$ and $R$ be any two unitary rings.
Let ${}_{R}\mathfrak{M}$ denote the category of left $R$-modules,
${}_{R}\mathfrak{M}_{\Lambda}$ the category of $(R,\Lambda)$-bimodules, and so on.
For a $(R,\Lambda)$-bimodule $M$, we let ${}_{R}M$ and $M_{\Lambda}$ denote $M$ considered as a left $R$-module and a right $\Lambda$-module, respectively. 
The rings $\Lambda$ and $R$ are said to be Morita equivalent if the categories ${}_{R}\mathfrak{M}$ and ${}_{\Lambda}\mathfrak{M}$ are equivalent
(this is the case if and only if the categories $\mathfrak{M}_{R}$ and $\mathfrak{M}_{{\Lambda}}$ are equivalent).
For example, if $n \in \N$ and $\Lambda=M_{n \times n}(R)$ then it is well-known that $\Lambda$ and $R$ are Morita equivalent.
We shall recall and use some basic facts on Morita equivalence, and refer the reader
to \cite[\S 3D]{MR632548}, \cite[Chapter 4]{MR1972204}, \cite[Chapter 7]{MR1653294} or \cite[\S 22]{MR1245487} for further details.

A generator $P$ for $\mathfrak{M}_{R}$ is a right $R$-module such that every right $R$-module is an epimorphic image of
$\oplus_{i \in I} P$ for $I$ sufficiently large, or equivalently, if $R_{R}$ is a direct summand of $P^{n}$ for some $n \in \N$.
A progenerator for $\mathfrak{M}_{R}$ is a finitely generated projective generator.
A progenerator for $_{\Lambda}\mathfrak{M}$ is defined analogously. 

A $(\Lambda,R)$-progenerator is a $(\Lambda,R)$-bimodule $P$ such that $P_{R}$ is a progenerator for $\mathfrak{M}_{R}$
and the canonical ring homomorphism $\Lambda \rightarrow \End(P_{R})$, $\lambda \mapsto (p \mapsto \lambda p)$ is an isomorphism.
In this case, $_{\Lambda}P$ is a progenerator for ${}_{\Lambda}\mathfrak{M}$ and the canonical ring homomorphism 
$R^{\mathrm{op}} \rightarrow \End({}_{\Lambda}P)$, $r \mapsto (p \mapsto pr)$ is an isomorphism.
For such a $P$ we let $Q=P^{*}=\Hom_{R}(P_{R},R_{R})$ be the $R$-linear dual of $P$.
Note that $Q$ is an $(R,\Lambda)$-bimodule with
left action of $R$ defined by $(rq)p=r(qp)$ and  right action of $\Lambda$ defined by $(q\lambda)p = q(\lambda p)$,
where $r \in R$, $q \in Q$, $p \in P$ and $\lambda \in \Lambda$.
Moreover, $P^{*}=Q$ is in fact a $(R,\Lambda)$-progenerator.

Now suppose that $\Lambda$ and $R$ are Morita equivalent.
Then by a theorem of Morita (see \cite[Theorem 22.2]{MR1245487}) there exists a $(\Lambda,R)$-progenerator $P$ such that the functors
\begin{align*}
G : & {}_{R}\mathfrak{M} \longrightarrow {}_{\Lambda}\mathfrak{M}, \quad N \mapsto P \otimes_{R} N \\
F : &{}_{\Lambda}\mathfrak{M} \longrightarrow {}_{R}\mathfrak{M}, \quad M \mapsto P^{\ast} \otimes_{\Lambda} M
\end{align*}
are mutually inverse category equivalences.
Moreover, the isomorphism classes of category equivalences
${}_{R}\mathfrak{M} \longrightarrow {}_{\Lambda}\mathfrak{M}$ are in one-one correspondence 
with the isomorphism classes of $(\Lambda,R)$-progenerators (see \cite[Theorem 18.28]{MR1653294});
of course the same statement holds with $\Lambda$ and $R$ swapped.

\begin{remark}\label{rmk:when-iso-to-product-of-matrix-rings}
%Two rings $\Lambda$ and $R$ are Morita equivalent if and only if $\Lambda$ is isomorphic to $\End(P_{R})$ for some progenerator
%$P_{R}$ of $\mathfrak{M}_{R}$ (see \cite[Proposition 18.33]{MR1653294}). 
If $R$ is a ring over which all finitely generated projective modules are free (this is the case when $R$ is a local ring, for example)
then $\Lambda$ is Morita equivalent to $R$ if and only if $\Lambda$ is isomorphic to $M_{n \times n}(R)$ for some $n \in \N$
(see \cite[Corollary 18.36]{MR1653294}).
More generally, if $R=R_{1} \oplus \cdots \oplus R_{k}$ is a direct product of rings and each $R_{i}$ is a ring over which all finitely generated projective modules are free, then $\Lambda$ is Morita equivalent to $R$ if and only if $\Lambda$ is isomorphic to
$M_{n_{1} \times n_{1}}(R_{1}) \oplus \cdots \oplus M_{n_{k} \times n_{k}}(R_{k})$ for some $n_{1}, \ldots, n_{k} \in \N$.
\end{remark}

\subsection{Rings that are Morita equivalent to commutative rings}

We now specialise to the situation where $\Lambda$ is Morita equivalent to its centre $R:= \zeta(\Lambda)$.
As Morita equivalent rings have isomorphic centres (see \cite[Corollary 18.42]{MR1653294}),
this assumption is the same as supposing that $\Lambda$ is Morita equivalent to some commutative ring.
We recall the convention that all $\Lambda$-modules are assumed to be left modules unless stated otherwise.

\begin{definition}\label{def:Morita-equiv-Fitt}
	Let $M$ be a finitely presented $\Lambda$-module.
	Then we define the Fitting invariant of $M$ over $\Lambda$ to be the $R$-ideal
	\[
	\Fitt_{\Lambda}(M) := \Fitt_{R}(F(M)) = \Fitt_{R}(P^{\ast} \otimes_{\Lambda} M).
	\]
\end{definition}

\begin{remark}\label{rmk:compatible-Fitt-defs}
	Suppose that $\Lambda$ is both a Fitting order and Morita equivalent to its centre $R$.
	We claim that $\Fitt_{\Lambda}(M) = \Fitt_{\Lambda}^{\max}(M)$
	for every finitely generated $\Lambda$-module $M$ and so the two notions of Fitting invariant coincide in this setting. 
	To see this, note that since $\Lambda$ is a Fitting order, \cite[Example 23.3 and Theorem 23.11]{MR1838439}
	show that we can write $R=R_{1} \oplus \cdots \oplus R_{k}$ where each $R_{i}$ is a commutative local ring.
	Thus by Remark \ref{rmk:when-iso-to-product-of-matrix-rings}, $\Lambda$ is isomorphic to
	$M_{n_{1} \times n_{1}}(R_{1}) \oplus \cdots \oplus M_{n_{k} \times n_{k}}(R_{k})$ for some $n_{1}, \ldots, n_{k} \in \N$.
	The claim now follows by applying \cite[Proposition 3.4]{MR3092262} to each component.
\end{remark}

\begin{remark} 
Suppose that $\Lambda=M_{n \times n}(R)$ for some $n \in \N$ and some commutative ring $R$.
In this situation, Fitting invariants over $\Lambda$ are defined in \cite[\S 2]{MR3092262}
using an explicit version of the Morita equivalence of $\Lambda$
and $R$, and this definition agrees with Definition \ref{def:Morita-equiv-Fitt}.
Moreover, it is trivial to extend the definition of ibid.\ to the situation in which $\Lambda$ is a direct product of matrix rings over commutative
rings; again both definitions agree in this situation. However, as we shall see in Example \ref{ex:morita-not-matrix} below, 
a ring that is Morita equivalent to its centre need not be isomorphic to a product of matrix rings over commutative rings;
thus the results presented here extend those of \cite[\S 2]{MR3092262}.
\end{remark}

\begin{remark}\label{rmk:Fitt-over-R-vs-Lambda}
Suppose that $M$ is a $\Lambda$-module that is finitely presented over both $\Lambda$ and its centre $R$.
Then of course one can consider $\Fitt_{R}(M)$, but in general this is a coarser invariant than $\Fitt_{\Lambda}(M)$.
For example, if $\Lambda=M_{n \times n}(R)$ for some $n \in \N$ then $\Fitt_{R}(M)=\Fitt_{\Lambda}(M)^{n}$
by \cite[Theorem 2.2(viii)]{MR3092262}.
\end{remark}

\begin{example}\label{ex:morita-not-matrix}
Let $R$ be a Dedekind domain and let $\mathrm{Cl}(R)$ denote its class group. 
Suppose that there exists a non-zero ideal $\mathfrak{a}$ of $R$ such that the image of $\mathfrak{a}$ in
$\mathrm{Cl}(R)/\mathrm{Cl}(R)^{2}$ is non-trivial.
For example, we can take $\mathfrak{a}$ to be any non-principal ideal of $R=\Z[\sqrt{-5}]$.
Let
\[
\Lambda = \left( 
\begin{array}{cc}
R & \mathfrak{a} \\
\mathfrak{a}^{-1} & R
\end{array}
\right)
\]
be the ring of all $2 \times 2$ matrices $(x_{ij})$ where $x_{11}$ ranges over all elements of $R$,
$x_{12}$ ranges over all elements of $\mathfrak{a}$, etc.
Let $P = R \oplus \mathfrak{a}$ and note that this is a progenerator for $\mathfrak{M}_{R}$.
Then $\Lambda \simeq \End_{R}( P )$ and so $\Lambda$ is Morita equivalent to $R$
by \cite[Proposition 18.33]{MR1653294}. However, \cite[Exercise 4.2.5]{MR1773562}
shows that $\Lambda$ is \emph{not} isomorphic to $M_{2 \times 2}(R)$.
\end{example}

\begin{proposition}
	The Fitting invariant $\Fitt_{\Lambda}(M)$ is well-defined.
\end{proposition}

\begin{proof}
   	We first note that $F(M)$ is a finitely presented $R$-module since the property of being finitely presented is preserved under equivalence
	of module categories (see the discussion in \cite[\S 18A, p.\ 481]{MR1653294}, for instance).
	Now let $G'$ and $F'$ be another pair of mutually inverse category equivalences of ${}_{R}\mathfrak{M}$ and ${}_{\Lambda}\mathfrak{M}$.
	Then the compositions of functors
	$F' \circ G$ and $F \circ G'$ are mutually inverse category self-equivalences of ${}_{R}\mathfrak{M}$.
	Thus the result follows from Lemma \ref{lem:auto-equivs} below.
\end{proof}

\begin{lemma}\label{lem:auto-equivs}
	Let $T : {}_{R}\mathfrak{M} \longrightarrow {}_{R}\mathfrak{M}$ be any self-equivalence of ${}_{R}\mathfrak{M}$.
	If $M$ is any $R$-module then $\Ann_{R}(M) = \Ann_{R}(T(M))$.
	Moreover, $\Fitt_{R}(M) = \Fitt_{R}(T(M))$ if we further assume that $M$ is finitely presented.
\end{lemma}

\begin{proof}
By \cite[Corollary 18.29]{MR1653294} there exists an $(R,R)$-progenerator $W$ such that the functor
$T': {}_{R}\mathfrak{M} \longrightarrow {}_{R}\mathfrak{M}$, $M \mapsto W \otimes_{R} M$ is naturally isomorphic to $T$.
In particular, $T(M)$ and $T'(M)$ are isomorphic as $R$-modules.
Hence $\Ann_{R}(T(M))=\Ann_{R}(T'(M))$ and, if $M$ is finitely presented, $\Fitt_{R}(T(M))=\Fitt_{R}(T'(M))$.
Thus we can and do assume without loss of generality that $T=T'$.

Let $x \in \Ann_{R}(M)$. Then
\[
x \cdot T(M) = x \cdot (W \otimes_{R} M) = W \otimes_{R} (x \cdot M) = 0.
\]
Hence $\Ann_{R}(M) \subset \Ann_{R}(T(M))$. 
Moreover, there exists a functor $U: {}_{R}\mathfrak{M} \longrightarrow {}_{R}\mathfrak{M}$ such that $U \circ T$ is naturally isomorphic to 
the identity functor on ${}_{R}\mathfrak{M}$ and so the same argument gives
$\Ann_{R}(T(M)) \subset \Ann_{R}(UT(M)) = \Ann_{R}(M)$. Thus $\Ann_{R}(M)=\Ann_{R}(T(M))$.

For the second claim choose a finite presentation
$R^{a} \stackrel{h}\longrightarrow R^{b} \twoheadrightarrow M$.
As $\Fitt_{R}(M)$ is generated by the $b \times b$ minors of $h$, we can and do assume that $a=b$.
Hence we may view $h$ as an element of $M_{b \times b}(R)$.
Applying $T$ yields an endomorphism
\[
T(h) = 1 \otimes h \in  \End_{R}(W \otimes_{R}  R^{b}) \simeq  \End_{R}(W^{b}) \simeq M_{b \times b}(R),
\]
where the last isomorphism is induced by $\End_{R}(W) \simeq R$.
Thus we have
\[
\det (T(h)) = \det(1 \otimes h) = \det(h).
\]
This shows $\Fitt_{R}(M) \subset \Fitt_{R}(T(M))$ and we again obtain equality by symmetry.
\end{proof}

\begin{example}
	Let $R$ be a Dedekind domain and let $\mathfrak a$ be a non-zero
	(fractional) ideal of $R$. Then $\mathfrak a$ is an $(R,R)$-progenerator
	and $\mathfrak{a}^{\ast}$ naturally identifies with $\mathfrak{a}^{-1}$.
	Moreover, Lemma \ref{lem:auto-equivs} implies that
	$\Fitt_{R}(M)=\Fitt_{R}(\mathfrak{a} \otimes_{R} M)$ 
	for every finitely generated $R$-module $M$.
\end{example}

\begin{lemma} \label{lem:basic-props-Morita}
Let $M_{1}$, $M_{2}$, $M_{3}$ be finitely presented $\Lambda$-modules.
\begin{enumerate}
		\item 
		If $\pi: M_1 \twoheadrightarrow M_2$ is an epimorphism, then
		$\Fitt_{\Lambda}(M_1) \subseteq \Fitt_{\Lambda}(M_2)$.
		\item
		Fitting invariants over $\Lambda$ behave well under direct sums:
		\[
		\Fitt_{\Lambda}(M_1 \oplus M_3) = 
		\Fitt_{\Lambda}(M_1) \cdot \Fitt_{\Lambda}(M_3).
		\]
		\item
		If $M_1 \rightarrow M_2 \rightarrow M_3 \rightarrow 0$
		is an exact sequence, then
		\[
		\Fitt_{\Lambda}(M_1) \cdot \Fitt_{\Lambda}(M_3) 
		\subseteq \Fitt_{\Lambda}(M_2).
		\]
	\end{enumerate}
\end{lemma}

\begin{proof}
The equivalence of categories $F : {}_{\Lambda}\mathfrak{M} \longrightarrow {}_{R}\mathfrak{M}$ preserves epimorphisms,
direct sums and exact sequences by \cite[Propositions 21.2, 21.4 and 21.5]{MR1245487}.
Thus the results follow from the corresponding properties of Fitting ideals over $R$ (see Lemma \ref{lem:basic-props-comm}).
\end{proof}

\begin{lemma}
Let $M$ be a finitely presented $\Lambda$-module.  
Then for any homomorphism $R \rightarrow S$ of commutative rings,
$S \otimes_{R} M$ is a finitely presented $S \otimes_{R} \Lambda$-module and
\[
		\Fitt_{S \otimes_{R} \Lambda}(S \otimes_{R} M) = S \otimes_{R} \Fitt_{\Lambda}(M).
\]
\end{lemma}

\begin{proof}
The first claim follows by applying the right exact functor $S \otimes_{R} -$ to a finite presentation of $M$.
For the second claim, observe that since $P^{*}$ is a $(R,\Lambda)$-progenerator, it is straightforward to check from the definitions
that $S \otimes_{R} P^{*}$ is a $(S,S \otimes_{R} \Lambda)$-progenerator.
Thus $(S \otimes_{R} P^{*}) \otimes_{S \otimes_{R} \Lambda} -$ induces an equivalence of categories
${}_{S \otimes_{R} \Lambda} \mathfrak{M} \rightarrow {}_{S} \mathfrak{M}$
and so Definition~\ref{def:Morita-equiv-Fitt} and Remark \ref{rem:base-change} give
\begin{align*}
\Fitt_{S \otimes_{R} \Lambda}(S \otimes_{R} M)
& = \Fitt_{S}( (S \otimes_{R} P^{*}) \otimes_{S \otimes_{R} \Lambda} (S \otimes_{R} M )) \\
& = \Fitt_{S}( S \otimes_{R} (P^{*} \otimes_{\Lambda} M) ) \\
&= S \otimes_{R} \Fitt_{R}( P^{*} \otimes_{\Lambda} M ) \\
&= S \otimes_{R} \Fitt_{\Lambda}(M). \qedhere
\end{align*}

\end{proof}

\begin{proposition}
	We have $\Fitt_{\Lambda}(M) \subset \Ann_{R}(M)$.
\end{proposition}

\begin{proof}
	Let $x \in \Fitt_{\Lambda}(M) = \Fitt_{R}(F(M))$.
	Then $x$ annihilates $F(M) = P^{\ast} \otimes_{\Lambda} M$ by the corresponding property of Fitting ideals over the commutative ring $R$.
	Hence $x$ also annihilates $F(M)^{k}$ for every $k \in \N$.
	As $P^{\ast}$ is in particular a progenerator for $\mathfrak{M}_{\Lambda}$, there exists $n \in \N$ such that 
	$\Lambda_{\Lambda}$ is a direct summand of $(P^{\ast})^{n}$.
	Hence $M = \Lambda \otimes_{\Lambda} M$ occurs as a direct summand
	of $F(M)^{n} = (P^{\ast})^{n} \otimes_{\Lambda} M$ and thus is annihilated by $x$, as desired.
\end{proof}

Finally, we formulate an analogue of Example \ref{ex:R/I}.

\begin{proposition} \label{prop:quotient-by-I}
	Let $I$ be a two-sided ideal of $\Lambda$.
	Under the identification $\Lambda \simeq \End_{R}(P)$ we have
	$I = \Hom_{R}(P, \mathfrak a \cdot P)$ for a uniquely determined $R$-ideal $\mathfrak{a}$.
	Then we have an equality
	\[
	\Fitt_{\Lambda}(\Lambda / I) = \Fitt_{R}(\Hom_{R}(P, R / \mathfrak{a})).
	\]
\end{proposition}

\begin{proof}
	That $I$ is of the given form is \cite[Theorem 16.14(v)]{MR1972204}.
	The canonical short exact sequence
	$0 \rightarrow I \rightarrow \Lambda \rightarrow \Lambda / I \rightarrow 0$ yields a short exact sequence
	\[
	0 \rightarrow P^{\ast} \otimes_{\Lambda} I \rightarrow P^{\ast} \rightarrow P^{\ast} \otimes_{\Lambda} \Lambda / I \rightarrow 0.
	\]
	We claim that the image of $P^{\ast} \otimes_{\Lambda} I $ in $P^{\ast}$ equals
	$\Hom_{R}(P, \mathfrak{a})$.
	In fact, the map
	\[
	P^{\ast} \otimes_{\Lambda} I = \Hom_{R}(P,R) \otimes_{\Lambda}
	\Hom_{R}(P, \mathfrak{a} \cdot P) \longrightarrow P^{\ast} = \Hom_{R}(P,R)
	\]
	is given by $f \otimes g \mapsto f \circ g$, where $f \in \Hom_{R}(P,R)$ and
	$g \in \Hom_{R}(P, \mathfrak{a} \cdot P)$.
	As the image of $g$ lies in $\mathfrak{a} \cdot P$ and $f$ is $R$-linear, the image of $f \circ g$ actually lies in $\Hom_{R}(P, \mathfrak{a})$.
	Thus we in fact have a map
	\[
	\alpha: P^{\ast} \otimes_{\Lambda} I = \Hom_{R}(P,R) \otimes_{\Lambda} \Hom_{R}(P, \mathfrak a \cdot P) \longrightarrow \Hom_{R}(P,\mathfrak{a}).
	\]
	To show that $\alpha$ is surjective (and thus an isomorphism), it suffices to show the corresponding statement after localisation at each prime ideal $\mathfrak{p}$ of $R$.
	However, a projective module over a local ring is free, so there exists $n \in \N$
	such that $P \simeq R^{n}$.
	Via this isomorphism, both the domain and codomain of $\alpha$
	identify naturally with $\oplus_{i=1}^{n} \mathfrak{a}$ and $\alpha$ becomes the identity map.
	Hence $\alpha$ is an isomorphism.
	
	We have shown that we have an isomorphism of $R$-modules
	\[
	P^{\ast} \otimes_{\Lambda} \Lambda / I \simeq \Hom_{R}(P,R) / \Hom_{R}(P,\mathfrak{a}) \simeq
	\Hom_{R}(P, R / \mathfrak{a}),
	\]
	where the last isomorphism holds by projectivity of $P$. The result now follows.
\end{proof}

\begin{example}
	Suppose that $P$ is free of rank $n$ over $R$ for some $n \in \N$.
	Then $\Lambda$
	naturally identifies with the matrix ring $M_{n \times n}(R)$.
	Moreover, $I = M_{n \times n}(\mathfrak a)$ and so
	$\Lambda/ I = M_{n \times n}(R/\mathfrak a)$. Thus we have
	$\Fitt_{\Lambda}(\Lambda/I) = \mathfrak a^n$ and so Proposition
	\ref{prop:quotient-by-I} recovers 
	\cite[Theorem 2.2(ix)]{MR3092262}.
	One can also view $\Lambda/I$ as a module over the centre $R$,
	but $\Fitt_{R}(\Lambda/ I) = \mathfrak{a}^{n^2}$ is properly contained in
	$\Fitt_{\Lambda}(\Lambda/I)$ if $n>1$ (also see Remark \ref{rmk:Fitt-over-R-vs-Lambda}).
\end{example}

\subsection*{Acknowledgements}
The first named author was supported by EPSRC First Grant EP/N005716/1 
`Equivariant Conjectures in Arithmetic'.
The second named author acknowledges financial support provided by the 
Deutsche Forschungsgemeinschaft (DFG) 
within the Heisenberg programme (No.\, NI 1230/3-1).

\nocite*
\bibliography{fitting-survey-bib}{}
\bibliographystyle{amsalpha}

\end{document}